\documentclass[final,leqno,onefignum,onetabnum]{siamltex1213}

\usepackage{amsmath} 
\usepackage{amssymb}

\usepackage{enumerate}

\usepackage{tikz}
\usetikzlibrary{arrows}

\newif\ifAndo
\Andofalse

\newtheorem{example}[theorem]{Example}
\newtheorem{remark}[theorem]{Remark}

\usepackage{csquotes}  
\newcommand\q{\enquote}

\newcommand \D   {\text{D}}

\newcommand  \esssup {\mathop{\text{ess} \sup} }

\newcommand \eps   {\varepsilon}

\newcommand \N   {\mathbb{N}}
\newcommand \R   {\mathbb{R}}
\newcommand \C   {\mathbb{C}}

\newcommand \K   {\mathcal{K}}
\newcommand \Kinf{\mathcal{K_\infty}}

\newcommand \KL  {\mathcal{KL}}

\newcommand \LL  {\mathcal{L}}

\newcommand{\Uc}{\ensuremath{\mathcal{U}}}

\newcommand{\scalp}[2]{ \lel #1, #2 \rir }

\newcommand{\lel}{\left\langle}
\newcommand{\rir}{\right\rangle}

\newcommand \qrq   {\quad\Rightarrow\quad}
\newcommand \srs   {\ \ \Rightarrow\ \ }

\newcommand \Iff   {\Leftrightarrow}

\newcommand \re  {\mathrm{Re}}
\newcommand \im  {\mathrm{Im}}

\newcounter{syscounter}
\newenvironment{sysnum}{\begin{list}{($\Sigma{\arabic{syscounter}}$)}%
{\settowidth{\labelwidth}{($\Sigma4$)}
\settowidth{\leftmargin}{($\Sigma4$)~}%
\usecounter{syscounter}}}
{\end{list}}

\newcommand{\midset}{\;:\;}

\usepackage{xifthen}

\newcommand{\einsnorm}[2]{\ensuremath{
    \!\!\;\!\!\!\;
    \left\bracevert\!\!\!\!\!\left\bracevert
    \!
		\ifthenelse{\isempty{#2}}{#1}{#1(#2)}
    \!
      \right\bracevert\!\!\!\!\!\right\bracevert
    \!\!\;\!\!\!\;
  }}

\title{
Noncoercive Lyapunov functions for input-to-state stability of infinite-dimensional systems
\thanks{
This is a modified version of the SICON article published in 2020 that includes corrections from 
"Corrigendum: Noncoercive Lyapunov functions for input-to-state stability of infinite-dimensional systems", 
which was accepted to SICON journal in 2022. The numeration in this version of the article is the same as in the original article from 2020.\newline 
This work is partially supported by German Research Foundation (DFG), grant: MI 1886/2-1.\newline
A short conference version of this paper was presented at the 57th Conference on Decision and Control as:
B. Jacob, A. Mironchenko, J. R. Partington and F. Wirth. Remarks on input-to-state stability and non-coercive Lyapunov functions. Proc. of the 57th IEEE Conference on Decision and Control (CDC 2018), Miami Beach, USA, pp. 4803–4808, 2018.
}
}

\author{Birgit Jacob \thanks{B. Jacob is with Functional analysis group, School of Mathematics and Natural Sciences,
        University of Wuppertal, D-42119 Wuppertal, Germany (\email{jacob@math.uni-wuppertal.de}).}
\and
Andrii Mironchenko \thanks{Andrii Mironchenko is with Faculty of Computer Science and Mathematics, University of
  Passau, Germany (\email{andrii.mironchenko@uni-passau.de}).  Corresponding author.}
\and
Jonathan R.~Partington \thanks{J. R.~Partington is with School of Mathematics, University of Leeds, Leeds LS2 9JT, Yorkshire, United Kingdom (\email{j.r.partington@leeds.ac.uk}).}
\and
Fabian R.~Wirth \thanks{F. R.~Wirth is with 
Faculty of Computer Science and Mathematics, University of Passau,
94030 Passau, Germany (\email{fabian.(lastname)@uni-passau.de}).}
}

\date{\today}

\begin{document}
\maketitle

\begin{abstract}
We consider an abstract class of infinite-dimensional dynamical systems
with inputs. For this class, the significance of noncoercive Lyapunov
functions is analyzed. It is shown that the existence of such Lyapunov
functions implies norm-to-integral input-to-state stability. This property in turn is equivalent to input-to-state stability, if the system satisfies certain mild regularity assumptions. 
For a particular class of linear systems with unbounded
admissible input operators, explicit constructions
of noncoercive Lyapunov functions are provided. The theory is applied to a heat
equation with Dirichlet boundary conditions. 
\end{abstract}



\pagestyle{myheadings}
\thispagestyle{plain}
\markboth{Non-coercive Lyapunov functions for ISS of infinite-dimensional systems}
{Non-coercive Lyapunov functions for ISS of infinite-dimensional systems}

\section{Introduction}

The concept of input-to-state stability (ISS), introduced in \cite{Son89} for ordinary differential equations (ODEs), unifies the classical Lyapunov and input-output stability theories and has broad applications in nonlinear control theory, in particular to robust stabilization of nonlinear systems \cite{FrK08}, design of nonlinear observers \cite{ArK01}, analysis of large-scale networks \cite{DRW10, JTP94}, etc.

The influence of finite-dimensional ISS theory and a desire to develop powerful tools for robust control of linear and nonlinear distributed parameter systems resulted in extensions of ISS concepts to broad classes of infinite-dimensional systems, including partial differential equations (PDEs) with distributed and boundary controls, semilinear equations in Banach spaces, time-delay systems, etc. \cite{DaM13, JNP18, JLR08, KaK16b, KaK18, KaK19},  \cite{MiI15b, MiW18b, TPT18, Sch20}.

Currently ISS of infinite-dimensional systems is an active research area at the intersection of nonlinear control, functional analysis, Lyapunov theory and PDE theory, which brings such important techniques for stability analysis as characterizations 
of ISS and ISS-like properties in terms of weaker stability concepts 
\cite{JNP18, MiW18b, Sch18}, constructions of ISS Lyapunov functions for PDEs with distributed and boundary controls 
\cite{MiI15b, PrM12, TPT18, ZhZ18}, efficient methods for study of boundary control systems 
\cite{JNP18, JSZ17, KaK16b, KaK19, ZhZ19}, etc.
For a survey on ISS of infinite-dimensional systems, we refer to \cite{MiP19}.

It is a basic result in input-to-state stability theory that the existence of an ISS Lyapunov function implies ISS.
However, the construction of ISS Lyapunov functions for
infinite-dimensional systems is a challenging task, especially for systems with boundary inputs and/or for nonlinear systems. Already for undisturbed linear systems over Hilbert
spaces, \q{natural} Lyapunov function candidates constructed via solutions
of Lyapunov equations are of the form $V(x):= \lel Px,x \rir_X$, where
$\lel\cdot ,\cdot \rir_X$ is a scalar product in $X$ and $P$ is a
self-adjoint, bounded and positive linear operator, whose spectrum may contain $0$. In this case $V$
is not coercive  and satisfies only the  weaker property that $V(x)>0$ for $x\neq 0$.
Hence the question arises, whether such \q{non-coercive} Lyapunov functions
can be used to conclude that a given system is ISS.
A thorough study of a similar question related to characterizations of uniform global asymptotic stability has recently been performed in \cite{MiW19a, MiW19b}. 

In \cite[Section III.B]{MiW18b}, it was shown for a class of semilinear
equations in Banach spaces with Lipschitz continuous nonlinearities that
the existence of a non-coercive Lyapunov function implies ISS provided the
flow of the system has some continuity properties with respect to states
and inputs at the origin, and the finite-time reachability sets of the
system are bounded. However, this class of systems does not include many
important systems such as linear control systems with admissible inputs
operators, which are crucially important for the study of partial
differential equations with boundary inputs. 

In this paper, we extend the results from \cite[Section III.B]{MiW18b} to a broader class of systems, which includes at least some important classes of boundary control systems. The characterizations of ISS developed in \cite{MiW18b} will play a central role in these developments.

We start by defining a general class of control systems in
Section~\ref{sec:Prelim}. This class covers a wide range of
infinite-dimensional systems. For this class, several stability concepts
are defined which relate to the characterizations of ISS, in particular to
the characterization with the help of noncoercive Lyapunov functions. 
We define also several further ISS-like properties, which we call norm-to-integral ISS and integral-to-integral ISS. Integral-to-integral ISS has been studied in \cite{Son98} and it was shown that integral-to-integral ISS is equivalent to ISS for systems of ordinary differential equations with sufficiently regular right hand side $f$. Further relations of ISS and integral-to-integral ISS for ODE systems have been developed in \cite{GSW99, KeD16} and other works.


Although ISS is no longer equivalent to integral-to-integral ISS for infinite-dimensio\-nal systems, \emph{we prove in Theorem~\ref{thm:ncISS_LF_sufficient_condition_NEW} that ISS is equivalent to norm-to-integral ISS for a broad class of infinite-dimensional systems provided the flow of the system has some continuity properties w.r.t. states and inputs at the origin (CEP property) and the finite-time reachability sets of the system are bounded (BRS property)}. The proof of this criterion is performed in 3 steps. First, we recall that norm-to-integral ISS implies a so-called uniform limit property, as proved in \cite[Section III.B]{MiW18b}.  
Next we show that integral-to-integral ISS implies local stability of a control system provided the flow of the system is continuous w.r.t. state and inputs at the origin. This is done in Proposition~\ref{prop:ncISS_plus_CEP_implies_ULS}.
The third and final step in the proof of Theorem~\ref{thm:ncISS_LF_sufficient_condition_NEW} is the application of the ISS superposition theorem from \cite{MiW18b}.

In Section~\ref{sec:Non-coercive ISS Lyapunov theorem}, we introduce non-coercive ISS Lyapunov functions for abstract control systems, and show that \emph{existence of such a function for a forward-complete system implies norm-to-integral ISS (Proposition~\ref{prop:ncLF_implies_norm-to-integral-ISS}), and it implies ISS provided CEP and BRS properties are satisfied (Theorem~Theorem~\ref{t:ISSLyapunovtheorem}).} Derivations of these results rely on the characterizations of ISS obtained in 
Section~\ref{sec:Relations between ISS and norm-to-integral ISS}. 
In Section~\ref{sec:Remarks on the definition of the ISS Lyapunov function}, we discuss the employed definition of the ISS Lyapunov function for various common choices of input spaces.

\emph{In Section~\ref{sec:Lyapunov_Theorem_systems_unbounded_operators}, we derive
a constructive converse ISS Lyapunov theorem (Theorem~\ref{thm:Gen_ISS_LF_Construction}) for certain classes of linear systems with admissible input operators. In particular, our results can be applied for a broad class of analytic semigroups
over Hilbert spaces generated by subnormal operators, as discussed in Section~\ref{sec1}.}

It is well-known that the classic heat equation with Dirichlet boundary inputs is ISS, which has been verified by means of several different methods: \cite{JNP18, KaK16b, MKK19}. However, no constructions for ISS Lyapunov functions have been proposed. 
In Section~\ref{sec:Applications}, we show that using the constructions developed in 
Theorem~\ref{thm:Gen_ISS_LF_Construction}, one can construct a non-coercive ISS Lyapunov function for this system.
It is still an open question, whether a coercive ISS Lyapunov function for a heat equation with the Dirichlet boundary input  exists (note, that for the system with Neumann boundary input a coercive quadratic ISS Lyapunov function can be constructed, see \cite{ZhZ18}).

\textbf{Notation:} We use the following notation. The nonnegative reals are denoted by $\R_+:=[0,\infty)$. The open ball of
radius $r$ around $0$  in a normed vector space $X$ is denoted by $B_r:=B_{r,X}:=\{x \in X: \|x\|_X <
r\}$. 
Similarly, \mbox{$B_{r,\Uc}:=\{u \in \Uc: \|u\|_\Uc < r\}$}. 
By $\mathop{\overline{\lim}}$ we denote the  limit superior.
For any normed linear space $X$, for any $S \subset X$ we denote
the closure of $S$ by 
$\overline{S}$. 
For a linear operator $A:X\to X$ in a Hilbert space $X$ (bounded or densely defined unbounded), we denote by $A^*$ the adjoint of the operator $A$.

For a function $u:\R_+\to U$, where $U$ is any set, we denote by 
$u|_{[0,t]}$ the restriction of $u$ to the interval $[0,t]$, that is
$u|_{[0,t]}:[0,t]\to U$ and $u|_{[0,t]}(s) = u(s)$ for all $s\in[0,t]$.

Let $U$ be a Banach space, $I$ be a closed subset of $\R$ and $p\in[1,+\infty)$. We define the following spaces (see \cite[Definition A.1.14]{JaZ12} for details):
\begin{eqnarray*}
M(\R_+,U) &:=& \{f: \R_+ \to U: f \text{ is strongly measurable} \},\\
L^p(\R_+,U) &:=& \{f \in M(\R_+,U): \|f\|_{L^p(\R_+,U)}:=\Big(\int_0^\infty\|f(s)\|^p_Uds\Big)^{1/p} < \infty \},\\
L^\infty(\R_+,U) &:=& \{f \in M(\R_+,U): \|f\|_{L^\infty(\R_+,U)}:=\esssup_{s\in \R_+}\|f(s)\|_U < \infty \}.
\end{eqnarray*}
Identifying the functions, which differ on a set of Lebesgue measure zero, the spaces $L^p(\R_+,U)$, $p\in[1,+\infty]$ are Banach spaces.

For the formulation of stability properties, the following classes of comparison functions are useful:
\begin{equation*}
\begin{array}{ll}
{\K} &:= \left\{\gamma:\R_+ \to \R_+ : \gamma\mbox{ is continuous and strictly increasing, }\gamma(0)=0\right\}\\
{\K_{\infty}}&:=\left\{\gamma\in\K:\ \gamma\mbox{ is unbounded}\right\}\\
{\LL}&:=\left\{\gamma:\R_+ \to \R_+:\ \gamma\mbox{ is continuous and strictly decreasing with}
 \lim\limits_{t\rightarrow\infty}\gamma(t)=0 \right\}\\
{\KL} &:= \left\{\beta: \R_+^2 \to \R_+ :\ \beta(\cdot,t)\in{\K},\ \forall t \geq 0,\  \beta(r,\cdot)\in {\LL},\ \forall r >0\right\}
\end{array}
\end{equation*}


\section{Preliminaries}
\label{sec:Prelim}

We begin by  defining (time-invariant) forward complete control systems
evolving on a Banach space $X$.

\begin{definition}
\label{Steurungssystem}
Let  $(X,\|\cdot\|_X)$, $(U,\|\cdot\|_U)$ be Banach spaces and $\Uc \subset \{f:\R_+ \to U\}$ be a normed vector space 
          satisfying the following two axioms:
                    
\emph{Axiom of shift invariance:} For all $u \in \Uc$ and all $\tau\geq0$ we have  $u(\cdot + \tau)\in\Uc$ with \mbox{$\|u\|_\Uc \geq \|u(\cdot + \tau)\|_\Uc$}.

\emph{Axiom of concatenation:} For all $u_1,u_2 \in \Uc$ and for all $t>0$ the concatenation of $u_1$ and $u_2$ at time $t$
\begin{equation}
u(\tau) := 
\begin{cases}
u_1(\tau), & \text{ if } \tau \in [0,t], \\ 
u_2(\tau-t),  & \text{ otherwise},
\end{cases}
\label{eq:Composed_Input}
\end{equation}
belongs to $\Uc$.
Consider a map $\phi:\R_+ \times X \times \Uc \to X$.

The triple $\Sigma=(X,\Uc,\phi)$ is called a \emph{forward complete control system}, if the following properties hold:

\begin{sysnum}
    \item\label{axiom:Identity} \emph{Identity property}: for every $(x,u) \in X \times \Uc$
          it holds that $\phi(0, x,u)=x$.
    \item \emph{Causality}: for every $(t,x,u) \in \R_+ \times X \times
          \Uc$, for every $\tilde{u} \in \Uc$ with $u|_{[0,t]} =
          \tilde{u}|_{[0,t]}$ it holds that $\phi(t,x,u) = \phi(t,x,\tilde{u})$.

    \item \label{axiom:Continuity} \emph{Continuity}: for each $(x,u) \in X
      \times \Uc$ the map $t \mapsto \phi(t,x,u),\ t\in [0,\infty)$ is continuous.
        \item \label{axiom:Cocycle} \emph{Cocycle property}: for all $t,h \geq 0$, for all
                  $x \in X$, $u \in \Uc$ we have
\begin{center}
$\phi(h,\phi(t,x,u),u(t+\cdot))=\phi(t+h,x,u)$.
\end{center}
\end{sysnum}
The space $X$ is called the \emph{state space}, $\Uc$ the \emph{input space} and $\phi$ the \emph{transition map}.
\end{definition}
This class of systems encompasses control systems generated by ordinary
differential equations (ODEs), switched systems, time-delay systems,
evolution partial differential equations (PDEs), abstract differential
equations in Banach spaces and many others. 

%
%

We single out two particular cases which will be of interest.

\begin{example}\emph{(Semilinear systems with Lipschitz nonlinearities)}.
Let $A$ be the generator of a strongly continuous semigroup (also called $C_0$-semigroup) $(T(t))_{t\ge 0}$ of bounded linear operators on $X$ and let $f:X\times U \to X$. Consider the system
\begin{equation}
\label{InfiniteDim}
\dot{x}(t)=Ax(t)+f(x(t),u(t)), \quad u(t) \in U,
\end{equation}
where $x(0)\in X$.
We study mild solutions of \eqref{InfiniteDim}, i.e. solutions $x:[0,\tau] \to X$ of the integral equation
\begin{align}
\label{InfiniteDim_Integral_Form}
x(t)=T(t)x(0) + \int_0^t T(t-s)f(x(s),u(s))ds,\quad t \in[0,\tau],
\end{align}
belonging to the space of continuous functions $C([0,\tau],X)$ for some $\tau>0$.

%
%

For system \eqref{InfiniteDim}, we use the following assumption concerning the nonlinearity $f$:
\begin{enumerate}[(i)]  
    \item  $f:X \times U \to X$ is Lipschitz continuous on bounded
subsets of $X$, i.e.  for
all $C>0$, there exists a $L_f(C)>0$, such that for all $ x,y \in B_C $ and for all $v \in B_{C,U}$, it holds that
\begin{eqnarray}
\|f(x,v)-f(y,v)\|_X \leq L_f(C) \|x-y\|_X.
\label{eq:Lipschitz}
\end{eqnarray}
    \item $f(x,\cdot)$ is continuous for all $x \in X$ and $f(0,0)=0$.
\end{enumerate}

%

Let $\Uc:=PC_b(\R_+,U)$ be the space of piecewise continuous functions, which are bounded and right-continuous, endowed with the supremum norm: $\|u\|_\Uc:=\sup_{t\geq 0}\|u(t)\|_U$. Then  our assumptions on $f$ ensure that mild
solutions of initial value problems of the form \eqref{InfiniteDim} exist and are
unique, according to
\cite[Proposition 4.3.3]{CaH98}. For system \eqref{InfiniteDim},
forward completeness is a further assumption. 
If these mild solutions exist on $[0,\infty)$ for every $x(0)\in X$ and $u\in PC_b(\R_+,U)$, then
$(X,PC_b(\R_+,U),\phi)$, defines a forward complete control system, where $\phi(t,x(0),u)$ denotes
the mild solution at time $t$.
\end{example}

\begin{example} \emph{(Linear systems with admissible control operators)}.
\label{exam2}
Consider linear systems of the form
\begin{equation}
\label{InfiniteDim2}
\dot{x}(t)=Ax(t)+ Bu(t), \quad x(0)\in X,\ t\ge 0,
\end{equation}
where  $A$ is the generator of a $C_0$-semigroup $(T(t))_{t\ge 0}$ on a Banach space $X$  and $B\in L(U,X_{-1})$ for some Banach space $U$.
Here $X_{-1}$ is the completion of $X$ with respect to the norm 
$ \|x\|_{X_{-1}}= \|(\beta I -A)^{-1}x\|_X$ 
for some $\beta\in \mathbb C $ in the resolvent set $\rho(A)$ of $A$. The semigroup  $(T(t))_{t\ge 0}$ extends uniquely to a $C_{0}$-semigroup  $(T_{-1}(t))_{t\ge 0}$ on $X_{-1}$ whose generator $A_{-1}$ is an extension of $A$, see e.g.\ \cite{EnN00}.
Thus, we may consider equation \eqref{InfiniteDim2} on the Banach space $X_{-1}$.  
 For every $x_0 \in X$ and every $u\in L^1_{\rm loc}([0,\infty),U)$, the function $x:[0,\infty)\rightarrow X_{-1}$,
\[ x(t):=T(t)x_0+\int_0^t T_{-1}(t-s)B u(s)ds,\quad t\ge 0,\]
is called a \emph{mild solution} of the system \eqref{InfiniteDim2}.

An operator $B\in L(U,X_{-1})$ is called a \emph{$q$-admissible control operator} for $(T(t))_{t\ge 0}$, where $1\le q\le \infty$, if
\[ \int_0^t T_{-1}(t-s)Bu(s)ds\in X\]
for every $t\ge 0$ and $u\in L^q([0,\infty),U)$ \cite{Wei89b}. If the operator $B\in L(U,X_{-1})$ is an $q$-admissible control operator for $(T(t))_{t\ge 0}$, then there exists for any $t\ge 0$ a constant $\kappa(t)>0$ such that 
\begin{eqnarray}
\left\|\int_0^t T_{-1}(t-s)Bu(s)\, ds\right\|_X \le \kappa(t) \|u\|_{q}, \quad u\in L^q([0,t),U),
\label{eq:Bounds-on-convolution}
\end{eqnarray}
see \cite{Wei89b}.

If $B$ is $\infty$-admissible and for every initial condition $x_0\in X$ and every input function $u\in L^\infty([0,\infty),U)$ the mild solution $x:[0,\infty)\rightarrow X$ is continuous, then 
%
$(X, L^\infty([0,\infty),U), \phi)$, where
\begin{eqnarray}
\phi (t, x_0,u):=T(t)x_0+\int_0^t T_{-1}(t-s)B u(s)ds,\quad t\geq 0, \ x_0\in X,\ u\in\Uc,
\label{eq:Linear-sys-solution-map}
\end{eqnarray}
defines a forward-complete control system as defined in Definition~\ref{Steurungssystem}.

\begin{remark}
\label{rem:Infty-admissibility} 
We note that,  $\infty$-admissibility of $B$ and  continuity of all mild solutions  $\phi(\cdot,x_0,u):[0,\infty)\rightarrow X$, with $x_0\in X$  and $u\in L^\infty([0,\infty),U)$ is implied by each of the following conditions:

\begin{itemize}
\item $B$ is $q$-admissible for some $q\in[1,\infty)$, see \cite[Proposition 2.3]{Wei89b},

\item  $B\in L(U,X_{-1})$, $\dim (U)<\infty$, $X$ is a Hilbert
  space and $A-\lambda I$ generates for a certain $\lambda\in\R$ an analytic semigroup that is similar  to a
  contraction semigroup, see \cite[Theorem 1]{JSZ17}.
\end{itemize}
\end{remark}
\end{example}


%
In this article, various stability concepts are needed for forward complete control systems.

\begin{definition}
\label{Assumption2}
Consider a forward complete control system $\Sigma=(X,\Uc,\phi)$.
\begin{enumerate}
\item
We call $0 \in X$ an \emph{equilibrium point (of the undisturbed system)}  if
$\phi(t,0,0) = 0$ for all $t \geq 0$.
\item
We say that $\Sigma$ is \emph{continuous at the equilibrium point (CEP)}, if $0$ is an equilibrium and    
for every $\eps >0$ and for any $h>0$ there exists a $\delta =
          \delta (\eps,h)>0$, so that 
\begin{eqnarray}
 t\in [0,h]\ \wedge\ \|x\|_X \leq \delta \ \wedge\  \|u\|_{\Uc} \leq \delta \qrq  \|\phi(t,x,u)\|_X \leq \eps.
\label{eq:RobEqPoint}
\end{eqnarray}
\item
We say that $\Sigma$ has \emph{bounded reachability sets (BRS)}, if for any $C>0$ and any $\tau>0$ it holds that 
\[
\sup\big\{
\|\phi(t,x,u)\|_X\! \midset \!\|x\|_X\leq C,\ \|u\|_{\Uc}\! \leq C,\ t \in [0,\tau]\big\} < \infty.
\]
\item
System $\Sigma$ is called
   \emph{uniformly locally stable (ULS)}, if there exist $ \sigma \in\Kinf$, $\gamma
          \in \Kinf$ and $r>0$ such that for all $ x \in \overline{B_r}$ and all $ u
          \in \overline{B_{r,\Uc}}$:
\begin{equation}
\label{GSAbschaetzung}
\left\| \phi(t,x,u) \right\|_X \leq \sigma(\|x\|_X) + \gamma(\|u\|_{\Uc}) \quad \forall t \geq 0.
\end{equation}
\item
We say that $\Sigma$ has the
  \emph{uniform limit property (ULIM)}, if there exists
    $\gamma\in\K$ so that for every $\eps>0$ and for every $r>0$ there
    exists a $\tau = \tau(\eps,r)$ such that 
for all $x$ with $\|x\|_X \leq r$ and all $u\in\Uc$ there is a $t\leq
\tau$ such that 
\begin{eqnarray}
\|\phi(t,x,u)\|_X \leq \eps + \gamma(\|u\|_{\Uc}).
\label{eq:ULIM_ISS_section}
\end{eqnarray}




%
\item \label{Def:ISS}
System $\Sigma$ is called \emph{(uniformly)  input-to-state stable (ISS)}, if there exist $\beta \in \KL$ and $\gamma \in \K$ 
such that for all $ x \in X$, $ u\in \Uc$ and $ t\geq 0$ it holds that
\begin {equation}
\label{iss_sum}
\| \phi(t,x,u) \|_{X} \leq \beta(\| x \|_{X},t) + \gamma( \|u\|_{\Uc}).
\end{equation}

\item \label{def:norm-i_ISS} 
We call $\Sigma$ \emph{norm-to-integral ISS} if there are $\alpha, \psi, \sigma\in\Kinf$ so that
for all $x\in X$, $u\in\Uc$ and $t\geq 0$ it holds that
\begin{equation}
\label{eq:ni-ISS}
\int_0^t \!\alpha(\|\phi(s,x,u)\|_X) ds \leq \psi(\|x\|_X) + t \sigma(\|u\|_{\Uc}).
\end{equation}
\end{enumerate}
\end{definition}

For the special case of $L^\infty$ input spaces we introduce one more stability notion 
\begin{definition}
\label{def:i-i_ISS} 
Consider a forward complete control system $\Sigma:=(X,\Uc,\phi)$ with the input space $\Uc:=L^\infty(\R_+,U)$, where $U$ is any normed linear space.

We call $\Sigma$ \emph{integral-to-integral ISS} if there are $\alpha, \psi, \sigma\in\Kinf$ so that
for all $x\in X$, $u\in\Uc$ and $t\geq 0$, it holds that
\begin{equation}
\label{eq:ii-ISS}
\int_0^t \!\alpha(\|\phi(s,x,u)\|_X) ds \leq \psi(\|x\|_X) +\! \int_0^t\! \sigma(\|u(s)\|_{U})ds.
\end{equation}
\end{definition}


\begin{remark}
\label{rem:RFC-REP} 
The CEP and BRS properties are motivated by the notions of a robust equilibrium point and of robust forward completeness, which were widely employed in \cite{KaJ11b}, see also \cite{MiW19a} where these concepts were used in the context of non-coercive Lyapunov theory.
\end{remark}


%

\begin{example} \emph{(Linear systems with admissible control operators)}
\label{exam2b}
We continue with Example \ref{exam2}, that is, we consider again equation \eqref{InfiniteDim2} and assume that
 $A$ generates a $C_0$-semigroup, $B\in L(U,X_{-1})$ is $\infty$-admissible and for every initial condition $x_0\in X$ and every input function $u\in L^\infty([0,\infty),U)$ the mild solution $x:[0,\infty)\rightarrow X$ is continuous.  These assumptions guarantee that $(X, L^\infty([0,\infty),U), \phi)$, where $\phi$ is given by \eqref{eq:Linear-sys-solution-map},  
defines a forward-complete control system. The system has the following properties
\begin{enumerate}
\item $0 \in X$ is an equilibrium point due to the linearity of the system,
\item $(X, L^\infty([0,\infty),U), \phi)$ has the CEP property,  and  bounded reachability sets (BRS), which follows easily from inequality \eqref{eq:Bounds-on-convolution} and linearity of the system.
\item $(T(t))_{t \ge 0}$ is exponentially stable if and only if  $(X, L^\infty([0,\infty),U), \phi)$ is ISS  \cite[Proposition 2.10]{JNP18}.
\item If $(T(t))_{t \ge 0}$ is exponentially stable, then $(X, L^\infty([0,\infty),U), \phi)$ has  the uniform limit property (ULIM), and is uniformly locally stable (ULS), which follows from the previous item.
\item \label{ii_ISS_and_ISS} 
If $\Sigma:=(X, L^\infty([0,\infty),U), \phi)$ has a so-called integral ISS property, 
then $\Sigma$ is ISS \cite{JNP18}.
To the best of the knowledge of the authors, it is unknown, whether or not
the converse implication holds for every linear system \eqref{InfiniteDim2}.
\end{enumerate}
\end{example}


\section{Characterization of ISS in terms of norm-to-integral ISS}
\label{sec:Relations between ISS and norm-to-integral ISS}

In this section, we characterize input-to-state stability in terms of the norm-to-integral ISS, which is interesting on its own right, but also it will be instrumental for the establishment of non-coercive ISS Lyapunov theorems in the next section.
We start with
\begin{proposition}
\label{prop:ISS-implies-norm-i-ISS} 
If a forward complete control system is ISS, then it is norm-to-integral ISS.
\end{proposition}

\begin{proof}
Let $\Sigma=(X,\Uc,\phi)$ be a forward complete ISS control system and let $\beta\in\KL$ be as in Definition~\ref{Assumption2}.
By Sontag's $\KL$-lemma \cite[Proposition 7]{Son98}, there are $\xi_1, \xi_2 \in\Kinf$ so that $\beta(r,t) \leq \xi_1^{-1}(e^{-t}\xi_2(r))$ for all $r,t\in\R_+$.
ISS of $\Sigma$ now implies that there is $\gamma\in\Kinf$ such that the following holds:
\begin {equation*}
\| \phi(t,x,u) \|_{X} \leq \xi_1^{-1}(e^{-t}\xi_2(\| x \|_{X})) + \gamma( \|u\|_{\Uc}),\quad t\geq 0, \ x\in X, \ u\in\Uc.
\end{equation*}
Define $\bar{\xi}(r):=\xi_1(\frac{1}{2}r)$, $r\in\R_+$. Using the inequality $\bar{\xi}(a+b)\leq \bar{\xi}(2a) + \bar{\xi}(2b)$, which is valid for all $a,b\in\R_+$, we obtain that for all $x\in X$, $u\in\Uc$ and $t\geq 0$ it holds that
\begin {equation}
\label{eq:ISS-implies-iiISS-auxiliary}
\bar{\xi}(\| \phi(t,x,u) \|_{X}) \leq e^{-t}\xi_2(\| x \|_{X}) + \xi_1(\gamma( \|u\|_{\Uc})).
\end{equation}
Integrating \eqref{eq:ISS-implies-iiISS-auxiliary}, we see that 
\begin {equation*}
\int_0^t\bar{\xi}(\| \phi(s,x,u) \|_{X}) ds \leq \xi_2(\| x \|_{X}) + t\xi_1 \circ \gamma( \|u\|_{\Uc}),\quad t\geq 0, \ x\in X, \ u\in\Uc,
\end{equation*}
which shows norm-to-integral ISS of $\Sigma$.
\end{proof}

Next we show that norm-to-integral ISS implies ISS for a class of forward-complete control systems satisfying the  CEP and BRS properties.
In order to prove this, we are going to use the following characterization of ISS, shown in \cite{MiW18b}:
\begin{theorem}
\label{thm:UAG_equals_ULIM_plus_LS}
Let $\Sigma=(X,\Uc,\phi)$ be a forward complete control system. Then $\Sigma$ is ISS if and only if $\Sigma$ is ULIM, ULS, and BRS.
\end{theorem}

In \cite[Proposition 8]{MiW18b} it was shown (with slightly different formulation, but the same proof) that
\begin{proposition}
\label{prop:ncISS_implies_ULIM} 
Let $\Sigma=(X,\Uc,\phi)$ be a forward complete control system. If $\Sigma$ is  norm-to-integral ISS, then $\Sigma$ is ULIM.
\end{proposition}

Next we provide a sufficient condition for the  ULS property. 
\begin{proposition}
\label{prop:ncISS_plus_CEP_implies_ULS} 
Let $\Sigma=(X,\Uc,\phi)$ be a forward complete control system satisfying the CEP property. If $\Sigma$ is norm-to-integral ISS, then $\Sigma$ is ULS.
\end{proposition}

\begin{proof}
Let $\Sigma$ be norm-to-integral ISS with the corresponding $\alpha,\psi,\sigma$ as in Definition~\ref{def:norm-i_ISS}.

By \cite[Lemma 2]{MiW18b}, $\Sigma$ is ULS if and only if for any $\eps>0$ there
is a $\delta>0$ such that 
\begin{equation}
\label{LS_Restatement}
\|x\|_X \leq\delta\ \wedge \ \|u\|_{\Uc} \leq \delta\ \wedge \  t\geq 0 \;\; \Rightarrow\;\;
\|\phi(t,x,u)\|_X \leq\eps.
\end{equation}
Seeking a contradiction, assume that $\Sigma$ is not ULS. Then there is $\varepsilon>0$ such that for any $\delta>0$ there are
 $x \in B_\delta$, $u\in B_{\delta,\Uc}$ and $t\geq 0$ such that $\|\phi(t,x,u)\|_X = \varepsilon$.
Then there are sequences $\{ x_k \}_{k\in\N}$ in $X$, $\{ u_k\}_{k\in\N}$ in $\Uc$, and $\{t_k\}_{k\in\N} \subset \R_+$ such that $x_k \to 0$ as $k \to \infty$, $u_k\to 0$ as $k\to\infty$ and
\begin{equation*}
    \| \phi(t_k,x_k,u_k) \|_X = \varepsilon \quad \forall k \geq 1.
\end{equation*}

Since $\Sigma$ is CEP, for the above $\eps$ there is a $\delta_1 = \delta_1(\eps,1)$ so that
\begin{eqnarray}
\|x\|_X \leq \delta_1 \ \wedge \ \|u\|_\Uc \leq \delta_1\ \wedge \ t\in [0,1] \srs \|\phi(t,x,u)\|_X < \eps.
\label{eq:ncLF_CEP_def_tilde_delta}
\end{eqnarray}
Define for $k\in\N$ the following time sequence:
\[
t^{1}_k := \sup\{t\in[0,t_k]: \|\phi(t,x_k,u_k)\|_X \leq \delta_1\},
\]
if the supremum is taken over a nonempty set, and set 
$t^{1}_k:= 0$ otherwise.

Again as $\Sigma$ is CEP, for the above $\delta_1$ there is a $\delta_2>0$ so that
\begin{eqnarray}
\|x\|_X\leq \delta_2\ \wedge \ \|u\|_\Uc \leq \delta_2\ \wedge \ t\in [0,1] \srs \|\phi(t,x,u)\|_X < \delta_1.
\label{eq:ncLF_CEP_def_bar_delta}
\end{eqnarray}
Without loss of generality, we assume that $\delta_2$ is chosen small enough so that
\begin{eqnarray}
\alpha(\delta_1) > \psi(\delta_2).
\label{eq:Relation_tilde_delta_and_bar_delta}
\end{eqnarray}
We now define
\[
t^{2}_k := \sup\{t\in[0,t_k]: \|\phi(t,x_k,u_k)\|_X \leq \delta_2\},
\]
if the supremum is taken over a nonempty set, and set $t^{2}_k:= 0$ otherwise.

Since $u_k\to 0$ and $x_k\to 0$ as $k\to\infty$, there is $K>0$ so that $\|u_k\|_\Uc \leq \delta_2$ and 
$\|x_k\|_X \leq \delta_2$ for $k\geq K$.

From now on, we always assume that $k\geq K$.

Using \eqref{eq:ncLF_CEP_def_tilde_delta}, \eqref{eq:ncLF_CEP_def_bar_delta} and the cocycle property,
it is not hard to show that for $k\geq K$ it must hold that $t_k\geq 2$,
as otherwise we arrive at a contradiction to 
$\|\phi(t_k,x_k,u_k)\|_X=\eps$.

Assume that $t_k - t^{1}_k <1$. This implies (since $t_k\geq 2$), that
$t^{1}_k>0$. By the cocycle property we have  
\begin{eqnarray*}
\|\phi(t_k,x_k,u_k)\|_X= \|\phi(t_k - t^{1}_k,\phi(t^{1}_k,x_k,u_k),u_k(\cdot + t^{1}_k)\|_X.
\end{eqnarray*}
The axiom of shift invariance justifies the inequalities
\[
\|u_k(\cdot + t^{1}_k)\|_\Uc \leq \|u_k\|_\Uc \leq \delta_2\leq \delta_1.
\]
Since $\|\phi(t^{1}_k,x_k,u_k)\|_X = \delta_1$, and $t_k - t^{1}_k <1$, we have by \eqref{eq:ncLF_CEP_def_tilde_delta} that 
$\|\phi(t_k,x_k,u_k)\|_X < \eps$, a contradiction.
Hence $t_k - t^{1}_{k}\geq 1$ for all $k\geq K$. 

Analogously, we obtain that  $t^{1}_k - t^{2}_k\geq 1$ and $t_k - t^{2}_k\geq 2$.

Define 
\[
x^{2}_k:=\phi(t^{2}_k,x_k,u_k),\quad u^{2}_k := u_k(\cdot + t^{2}_k)
\]
and
\[
x^{1}_k:=\phi(t^{1}_k,x_k,u_k),\quad u^{1}_k := u_k(\cdot + t^{1}_k).
\]
Due to the axiom of shift invariance $u^{1}_k, u^{2}_k \in\Uc$ and 
\[
\|u^{1}_k\|_\Uc \leq \|u^{2}_k\|_\Uc \leq \|u_k\|_\Uc \leq \delta_2.
\]
Also by the definition of $t^{2}_k$, we have $\|x^{2}_k\|_X =\delta_2$.

Applying \eqref{eq:ni-ISS}, we obtain for $t:=t_k - t^{2}_k$ that 
\begin{align}
\label{nc_LS_proof_Eq1}
\int_0^{t_k - t^{2}_k} \alpha(\|\phi(s,x^{2}_k,u^{2}_k)\|_X) ds &\leq \psi(\|x^{2}_k\|_X) + (t_k - t^{2}_k)\sigma(\|u^{2}_k\|_{\Uc}) \nonumber\\
&\leq \psi(\delta_2) + (t_k - t^{2}_k)\sigma(\|u_k\|_{\Uc}).
\end{align}
On the other hand, changing the integration variable and using the cocycle property, we obtain that
\begin{align*}
\int_0^{t_k - t^{2}_k}\hspace{-4mm} \alpha(\|\phi(s,x^{2}_k,u^{2}_k)&\|_X) ds
= \int_0^{t^{1}_k - t^{2}_k} \hspace{-4mm}\alpha(\|\phi(s,x^{2}_k,u^{2}_k)\|_X) ds +
\int_{t^{1}_k - t^{2}_k}^{t_k - t^{2}_k} \hspace{-4mm}\alpha(\|\phi(s,x^{2}_k,u^{2}_k)\|_X) ds\\
=& \int_0^{t^{1}_k - t^{2}_k} \hspace{-4mm}\alpha(\|\phi(s,x^{2}_k,u^{2}_k)\|_X) ds +
\int_{0}^{t_k - t^{1}_k} \hspace{-4mm}\alpha(\|\phi(s+ t^{1}_k - t^{2}_k,x^{2}_k,u^{2}_k)\|_X) ds\\
=& \int_0^{t^{1}_k - t^{2}_k} \hspace{-4mm}\alpha(\|\phi(s,x^{2}_k,u^{2}_k)\|_X) ds +
\int_{0}^{t_k - t^{1}_k} \hspace{-4mm}\alpha(\|\phi(s,x^{1}_k,u^{1}_k)\|_X) ds.
\end{align*}
For the last transition we have used that for every $s\in[0,t_k - t^{1}_k]$
\begin{align*}
\phi(s+ t^{1}_k - t^{2}_k,x^{2}_k,u^{2}_k)
&= \phi\big(s+ t^{1}_k - t^{2}_k,\phi(t^{2}_k,x_k,u_k),u_k(\cdot + t^{2}_k)\big)\\
&= \phi\big(s,\phi(t^{1}_k - t^{2}_k,\phi(t^{2}_k,x_k,u_k),u_k(\cdot + t^{2}_k)),u_k(\cdot + t^{1}_k)\big)\\
&= \phi\big(s,\phi(t^{1}_k,x_k,u_k),u_k(\cdot + t^{1}_k)\big)\\
&= \phi(s,x^{1}_k,u^{1}_k).
\end{align*}
By definition of $t^{2}_k$ and $t^{1}_k$, we have that 
\begin{align*}
\|\phi(s,x^{2}_k,u^{2}_k)\|_X  
&= \big\|\phi\big(s,\phi(t^{2}_k,x_k,u_k),u_k(\cdot + t^{2}_k)\big)\big\|_X \\
&= \|\phi(s+t^{2}_k,x_k,u_k)\|_X 
\geq \delta_2,\qquad s\in[0,t_k-t^{2}_k],
\end{align*}
and, analogously,
\[
\|\phi(s,x^{1}_k,u^{1}_k)\|_X \geq \delta_1,\quad s\in[0,t_k-t^{1}_k].
\]
Continuing the above estimates and using that $t_k - t^{1}_k\geq 1$ and $\alpha(\delta_1)>\alpha(\delta_2)$, we arrive at
\begin{align*}
\int_0^{t_k - t^{2}_k} \alpha(\|\phi(s,x^{2}_k,u^{2}_k)\|_X) ds 
&\geq (t^{1}_k - t^{2}_k)\alpha(\delta_2) + (t_k - t^{1}_k)\alpha(\delta_1)\\
&\geq (t^{1}_k - t^{2}_k)\alpha(\delta_2) + (t_k - t^{1}_k-1)\alpha(\delta_1) + \alpha(\delta_1)\\
&\geq (t_k - t^{2}_k - 1)\alpha(\delta_2) + \alpha(\delta_1).
\end{align*}
Since $t_k - t^{2}_k\geq 2$ and in view of \eqref{eq:Relation_tilde_delta_and_bar_delta}, we derive
\begin{eqnarray}
\label{nc_LS_proof_Eq2}
\int_0^{t_k - t^{2}_k} \alpha(\|\phi(s,x^{2}_k,u^{2}_k)\|_X) ds >
\frac{t_k - t^{2}_k}{2}\alpha(\delta_2) + \psi(\delta_2).
\end{eqnarray}
Combining inequalities \eqref{nc_LS_proof_Eq1} and \eqref{nc_LS_proof_Eq2}, we obtain
\begin{eqnarray*}
\frac{t_k - t^{2}_k}{2}\alpha(\delta_2)  < (t_k - t^{2}_k)\sigma(\|u_k\|_{\Uc}).
\end{eqnarray*}
This leads to
\begin{eqnarray*}
\frac{1}{2}\alpha(\delta_2)  < \sigma(\|u_k\|_{\Uc}), \quad k\geq K.
\end{eqnarray*}
Finally, since $\lim_{k\to\infty}\|u_k\|_\Uc= 0$, letting $k\to\infty$, we come to a contradiction.
This shows that $\Sigma$ is ULS.
\end{proof}

Now we combine the derived results into a criterion of ISS in terms of norm-to-integral ISS.
\begin{theorem}
\label{thm:ncISS_LF_sufficient_condition_NEW}
Let $\Sigma$ be a forward complete control system. Then $\Sigma$ is ISS if and only if $\Sigma$ is norm-to-integral ISS and has CEP and BRS properties. 
\end{theorem}

\begin{proof}
\q{$\Rightarrow$}. Clearly, ISS implies CEP and BRS properties. By Proposition~\ref{prop:ISS-implies-norm-i-ISS}, ISS implies  norm-to-integral ISS.

\q{$\Leftarrow$}.
Propositions~\ref{prop:ncISS_implies_ULIM} and \ref{prop:ncISS_plus_CEP_implies_ULS} imply that $\Sigma$ is ULIM and ULS.
Since $\Sigma$ is BRS, Theorem~\ref{thm:UAG_equals_ULIM_plus_LS} shows that $\Sigma$ is ISS.
\end{proof}

Some forward complete systems have the BRS and CEP properties intrinsically, which leads to the equivalence between ISS and norm-to-integral ISS for such classes of systems.
In particular, the following slight extension of \cite[Theorem 1]{Son98} holds:
\begin{corollary}
\label{cor:ni-ISS-and-ISS-ODEs} 
Let $X:=\R^n$ and $\Uc:=L^\infty(\R_+,\R^m)$. Consider ordinary differential equations of the form
\begin{eqnarray}
\dot{x} = f(x,u), \quad t>0,
\label{eq:ODEs}
\end{eqnarray}
with $f$ which is Lipschitz continuous in both variables. 
Define $\phi(\cdot,x,u)$ as the maximal solution of \eqref{eq:ODEs} with an initial condition $x \in \R^n$ and an input $u\in\Uc$.

Assume that \eqref{eq:ODEs} is forward complete. Then the following statements are equivalent:
\begin{itemize}
	\item[(i)] \eqref{eq:ODEs} is ISS
	\item[(ii)] \eqref{eq:ODEs} is integral-to-integral ISS
	\item[(iii)] \eqref{eq:ODEs} is norm-to-integral ISS
\end{itemize}
\end{corollary}
%

\begin{proof}
(i) $\Iff$ (ii). Was shown in \cite[Theorem 1]{Son98}. Here Lipschitz continuity in both variables is used, as the proof of this equivalence in 
\cite[Theorem 1]{Son98} is based on the smooth ISS converse Lyapunov theorem in \cite{SoW95}, which requires at least Lipschitz continuity in both variables (actually it was assumed that $f$ is continuously differentiable in $(x,u)$).

(i) $\Iff$ (iii). It is well-known that under made assumptions on $f$ a forward complete system \eqref{eq:ODEs} satisfies the BRS property (see \cite[Proposition 5.1]{LSW96}) and CEP property. The claim follows from Theorem~\ref{thm:ncISS_LF_sufficient_condition_NEW}.
\end{proof}

For linear systems with admissible control operators, we have:
\begin{proposition}
\label{prop:ISS-implies-iiISS-linear-systems} 
Let $X$ and $U$ be Banach spaces and let $\Uc:=L^\infty([0,\infty),U)$. Consider the system \eqref{InfiniteDim2} and assume that
$A$ generates a $C_0$-semigroup, $B\in L(U,X_{-1})$ is $\infty$-admissible and for every initial condition $x_0\in X$ and 
 every input $u\in \Uc$ the mild solution $x:[0,\infty)\rightarrow X$ of \eqref{InfiniteDim2} is continuous.

Then \eqref{InfiniteDim2} is ISS if and only if \eqref{InfiniteDim2} is norm-to-integral ISS.
\end{proposition}

\begin{proof}
"$\Rightarrow$". Follows by Proposition~\ref{prop:ISS-implies-norm-i-ISS}.

"$\Leftarrow$". As mentioned in Example~\ref{exam2b}, under the assumptions made \eqref{InfiniteDim2} is a well-posed control system satisfying CEP and BRS properties.
Theorem~\ref{t:ISSLyapunovtheorem} finishes the proof.

Alternative proof of "$\Leftarrow$". From norm-to-integral ISS of \eqref{InfiniteDim2} for $u\equiv 0$ the exponential stability of the semigroup generated by $A$ follows by means of a generalized Datko lemma \cite[Theorem 2]{Lit89}. As we assume that $B$ is $\infty$-admissible, ISS follows by \cite[Proposition 2.10]{JNP18}.
\end{proof}

\subsection{Remark on input-to-state practical stability}

In some cases, it is impossible (as in quantized control)
or too costly to construct a feedback that results in an ISS closed-loop system.
For these applications, one defines the following relaxation of the ISS property:
\begin{definition}
\label{Def:ISpS_wrt_set}
A control system $\Sigma=(X,\Uc,\phi)$ is called \emph{(uniformly) input-to-state practically stable (ISpS)}, if there exist $\beta \in \KL$, $\gamma \in \Kinf$ and $c>0$
such that for all $ x \in X$, $ u\in \Uc$ and $ t\geq 0$ the following holds:
\begin {equation}
\label{isps_sum}
\| \phi(t,x,u) \|_X \leq \beta(\| x \|_X,t) + \gamma( \|u\|_{\Uc}) + c.
\end{equation}
\end{definition}
The notion of ISpS has been proposed in \cite{JTP94} and has become very useful for 
control in the presence of quantization errors \cite{ShL12, JLR09}, sample-data control
\cite{NKK15} to name a few examples.

One of the requirements in Theorem~\ref{thm:ncISS_LF_sufficient_condition_NEW}
is that the CEP property holds. If this property is not available, we can still infer input-to-state practical stability of $\Sigma$, using the main result in \cite{Mir19a}.
\begin{theorem}
\label{thm:ncISpS_LF_sufficient_condition_NEW}
Let $\Sigma$ be a forward complete control system, which is BRS. If $\Sigma$ is norm-to-integral ISS, then $\Sigma$ is ISpS.
\end{theorem}

\begin{proof}
Proposition~\ref{prop:ncISS_implies_ULIM} implies that $\Sigma$ is ULIM.
Since $\Sigma$ is also BRS, \cite[Theorem III.1]{Mir19a} shows that $\Sigma$ is ISpS.
\end{proof}

\section{Non-coercive ISS Lyapunov theorem}
\label{sec:Non-coercive ISS Lyapunov theorem}

For a real-valued function $b:\R_+\to\R$ define the \emph{right-hand upper and lower Dini derivatives} at $t\in\R_+$ by
\begin{eqnarray*}
D^+b(t) := \mathop{\overline{\lim}} \limits_{h \rightarrow +0} \frac{b(t+h) - b(t)}{h},\qquad
D_+b(t) := \mathop{\underline{\lim}} \limits_{h \rightarrow +0} \frac{b(t+h) - b(t)}{h}
\end{eqnarray*}
respectively. Note that for all $b:\R_+\to\R$ and all $t\in\R_+$ it holds that 
\begin{eqnarray}
D^+b(t) = -D_+(-b(t)).
\label{eq:Dini-lower-upper}
\end{eqnarray}

Let $x \in X$ and $V$ be a
real-valued function defined in a neighborhood of $x$. The \emph{(right-hand upper) Dini derivative 
of $V$ at $x$ corresponding to the input $u \in\Uc$ along the trajectories of $\Sigma$} is defined by
\begin{equation}
\label{ISS_LyapAbleitung}
\dot{V}_u(x):=D^+ V\big(\phi(\cdot,x,u)\big)\Big|_{t=0}=\mathop{\overline{\lim}} \limits_{t \rightarrow +0} {\frac{1}{t}\Big(V\big(\phi(t,x,u)\big)-V(x)\Big) }.
\end{equation}

%


We need a lemma on derivatives of monotone functions.
\begin{lemma}
\label{lem:Integrals-of-monotone-functions} 
Let $b:\R_+\to\R$ be a nonincreasing function. Then for each $t\in\R_+$ it holds that
\begin{eqnarray}
b(t) \geq D^+ \int_0^t b(s) ds \geq D_+ \int_0^t b(s) ds \geq \lim_{h\to +0}b(t+h).
\label{eq:Lower-Dini-Derivative}
\end{eqnarray}
\end{lemma}

\begin{proof}
As $b$ is nonincreasing, it is Riemann integrable on any finite subinterval of $\R_+$, and thus the integrals in the formulation of the result make sense.
Pick any $t\geq 0$. By the definition of the Dini derivative and using
monotonicity, it holds that
\begin{eqnarray*}
D^+ \int_0^t b(s) ds 
= \mathop{\overline{\lim}} \limits_{h \rightarrow +0} \frac{1}{h}\Big({\int_0^{t+h} b(s) ds - \int_0^t b(s) ds}\Big)
&=& \mathop{\overline{\lim}} \limits_{h \rightarrow +0} \frac{1}{h}{\int_t^{t+h} b(s) ds}\\
&\leq& \mathop{\overline{\lim}} \limits_{h \rightarrow +0} \frac{1}{h}{\int_t^{t+h} b(t) ds} = b(t).
\end{eqnarray*}
On the other hand, we have that 
\begin{eqnarray*}
D_+ \int_0^t b(s) ds 
&\geq& \mathop{\underline{\lim}} \limits_{h \rightarrow +0} \frac{1}{h}{\int_t^{t+h} b(t+h) ds}
= \mathop{\underline{\lim}} \limits_{h \rightarrow +0} b(t+h).
\end{eqnarray*}
The inequality $D^+ \int_0^t b(s) ds \geq D_+ \int_0^t b(s)ds$ is clear.
\end{proof}

For stability analysis of nonlinear control systems, Lyapunov functions are an essential tool.
\begin{definition}
\label{def:noncoercive_ISS_LF}
Consider a control system $\Sigma = (X,\Uc,\phi)$.
A continuous function $V:X \to \R_+$ is called a \emph{non-coercive ISS Lyapunov function} for $\Sigma$, if there exist $\psi_2,\alpha \in \Kinf$ and $\sigma \in \K$ such that
\begin{equation}
\label{LyapFunk_1Eig_nc_ISS}
0 < V(x) \leq \psi_2(\|x\|_X), \quad \forall x \in X \setminus \{0\},
\end {equation}
and the Dini derivative of $V$ along the trajectories of $\Sigma$ for all $x \in X$ and $u\in \Uc$ satisfies
\begin{equation}
\label{DissipationIneq_nc}
\dot{V}_u(x) \leq -\alpha(\|x\|_X)  + \sigma(\|u\|_{\Uc}).
\end{equation}

Moreover, if \eqref{DissipationIneq_nc} holds just for $u=0$, we call $V$ a \emph{non-coercive Lyapunov function} for the undisturbed system $\Sigma$.
If additionally there is $\psi_1\in\Kinf$ so that the following estimate holds:
\begin{equation}
\label{LyapFunk_1Eig_LISS}
\psi_1(\|x\|_X) \leq V(x) \leq \psi_2(\|x\|_X), \quad \forall x \in X,
\end{equation}
then $V$ is called a \emph{coercive ISS Lyapunov function} for $\Sigma$.
\end{definition}

Note that continuity of $V$ and the estimate \eqref{LyapFunk_1Eig_nc_ISS} imply that $V(0)=0$.

The next proposition shows that the norm-to-integral ISS property arises naturally in the theory of ISS Lyapunov functions:
\begin{proposition}
\label{prop:ncLF_implies_norm-to-integral-ISS} 
Let $\Sigma=(X,\Uc,\phi)$ be a forward complete control system. Assume
that there exists a non-coercive ISS Lyapunov function for $\Sigma$. Then $\Sigma$ is norm-to-integral ISS.
\end{proposition}

\begin{proof}
Assume that $V$ is a non-coercive ISS Lyapunov function for $\Sigma$ with
the corresponding $\psi_2,\alpha,\sigma$.
Pick any $u\in\Uc$ and any $x\in X$. As we assume forward completeness of $\Sigma$, the trajectory $\phi(\cdot,x,u)$ exists for all $t\geq0$ and due to \eqref{DissipationIneq_nc}, we have for any $t>0$ that:
\begin{equation}
\label{DissipationIneq_nc-specified}
\dot{V}_{u(t+\cdot)}\big(\phi(t,x,u)\big) \leq -\alpha(\|\phi(t,x,u)\|_X)  + \sigma(\|u(t+\cdot)\|_{\Uc}).
\end{equation}
By definition of $\dot{V}$, and using the cocycle property for $\Sigma$, we have that
\begin{align*}
\dot{V}_{u(t+\cdot)}\big(\phi(t,x,u)\big)
&=\mathop{\overline{\lim}} \limits_{h \rightarrow +0} {\frac{1}{h}\Big(V\big(\phi(h,\phi(t,x,u),u(t+\cdot))\big)-V\big(\phi(t,x,u)\big)\Big) }\\
&=\mathop{\overline{\lim}} \limits_{h \rightarrow +0} {\frac{1}{h}\Big(V\big(\phi(t+h,x,u)\big)-V\big(\phi(t,x,u)\big)\Big) }.
\end{align*}
Defining $y(t):=V\big(\phi(t,x,u)\big)$, we see that 
\begin{eqnarray}
\dot{V}_{u(t+\cdot)}\big(\phi(t,x,u)\big) = D^+ y(t),
\label{eq:Lie-der-and-Dini-derivative}
\end{eqnarray}
and $y(0)= V(x)$ due to the identity axiom of the system $\Sigma$.

In view of the continuity axiom of $\Sigma$, for fixed $x,u$ the map $\phi(\cdot,x,u)$ is continuous, and thus $t\mapsto -\alpha(\|\phi(t,x,u)\|_X)$ is continuous as well.

For $t\geq 0$, define $G(t):=\int_0^t \alpha(\|\phi(s,x,u)\|_X) ds$ and
$b(t):=\sigma(\|u(t+\cdot)\|_{\Uc})$.
Note that by the axiom of shift invariance, $b$ is non-increasing.
As $G$ is continuously differentiable, we can rewrite the inequality \eqref{DissipationIneq_nc-specified} as
\begin{eqnarray}
D^+y(t) \leq - \frac{d}{dt}G(t) + b(t). 
\label{eq:Shorthand-notation-dissip-ineq}
\end{eqnarray}
Pick any $r>0$ and define $b(s)=b(0)$ for $s \in [-r,0]$. As $b$ is a
nonincreasing function on $[-r,\infty)$, it holds for any $t\geq 0$ that $b(t) \leq
\lim_{h\to +0}b(t-r+h)$, and by the final inequality in
Lemma~\ref{lem:Integrals-of-monotone-functions} applied to $b(\cdot - r)$,
we obtain for all $t\ge 0$
\[
b(t) \leq D_+ \int_0^{t} b(s-r) ds = - D^+ \Big(-\int_0^{t} b(s-r) ds\Big).
\]
Thus, \eqref{eq:Shorthand-notation-dissip-ineq} implies that 
\begin{eqnarray}
D^+y(t) + \frac{d}{dt}G(t) + D^+ \Big(-\int_0^{t} b(s-r) ds\Big) \leq 0.
\label{eq:Shorthand-notation-dissip-ineq-2}
\end{eqnarray}
Due to 
\[
D^+(f_1(t)+f_2(t))\leq D^+(f_1(t)) + D^+(f_2(t)),
\]
which holds for any functions $f_1, f_2$ on the real line, this implies that 
\begin{eqnarray*}
D^+\Big(y(t) + G(t) -\int_0^{t} b(s-r) ds\Big) \leq 0.
\end{eqnarray*}
It follows from 
\cite[Theorem 2.1]{Sza65} that $t\mapsto y(t) + G(t) -\int_0^{t} b(s-r)
ds$ is nonincreasing. As $G(0) = 0$, it follows that for all $r>0$ and all $t\ge 0$
\begin{eqnarray*}
y(t) + G(t) -\int_0^{t} b(s-r) ds \leq y(0) = V(x). 
\end{eqnarray*}
As $b$ is bounded, we may pass to the limit $r \to 0$ and obtain
\begin{eqnarray*}
y(t) + G(t) -\int_0^{t} b(s) ds \leq y(0) = V(x). 
\end{eqnarray*}
Now $y(t)\geq 0$ for all $t\in\R_+$, and so for all $t\ge 0$, $x \in X$, and $u\in\Uc$
\begin{align}
\label{eq:old-int-to-int-ISS}
\int_0^t \alpha(\|\phi(s,x,u)\|_X) ds 
&\leq \psi_2(\|x\|_X) + \int_0^{t} \sigma(\|u(s+\cdot)\|_{\Uc}) ds\\
&\leq \psi_2(\|x\|_X) + t \sigma(\|u\|_{\Uc}).\nonumber
\end{align}
This completes the proof.
\end{proof}

We can now state our main result on noncoercive ISS Lyapunov functions.

\begin{theorem}{}
    \label{t:ISSLyapunovtheorem}
Let $\Sigma=(X,\Uc,\phi)$ be a forward complete control system with the input space 
$\Uc:=L^\infty(\R_+,U)$, which is CEP and BRS.    
If there exists a (noncoercive) ISS Lyapunov function for $\Sigma$, then $\Sigma$ is ISS.    
\end{theorem}

\begin{proof}
By Proposition~\ref{prop:ncLF_implies_norm-to-integral-ISS}, $\Sigma$ is is norm-to-integral ISS. The application of Theorem~\ref{thm:ncISS_LF_sufficient_condition_NEW} finishes the proof.
\end{proof}

\subsection{Remarks on the definition of the ISS Lyapunov function}
\label{sec:Remarks on the definition of the ISS Lyapunov function}

In this section, we discuss the definition of the ISS Lyapunov function, which we adopted in this paper.

For any $u\in\Uc$ and any $\tau\geq 0$, define
$
u_\tau(s):=
\begin{cases}
u(s) &, \text{ if } s\in[0,\tau], \\ 
0 &, \text{ if } s> \tau.
\end{cases}
$

We start with a restatement of the ISS Lyapunov function concept.
\begin{lemma}
\label{lem:noncoercive_ISS_LF-inf-restatement}
Assume that for any $u\in\Uc$ it holds that $\inf_{\tau > 0}\|u_\tau\|_\Uc \leq \|u\|_\Uc$.

Then a continuous function $V:X \to \R_+$ is a \emph{non-coercive ISS Lyapunov function} for the system $\Sigma = (X,\Uc,\phi)$,  if and only if there exist $\psi_2,\alpha \in \Kinf$ and $\sigma \in \K$ such that \eqref{LyapFunk_1Eig_nc_ISS} holds
and the Dini derivative of $V$ along the trajectories of $\Sigma$ for all $x \in X$ and $u\in \Uc$ satisfies
\begin{equation}
\label{DissipationIneq_nc-inf}
\dot{V}_u(x) \leq -\alpha(\|x\|_X)  + \sigma(\inf_{\tau > 0}\|u_\tau\|_{\Uc}).
\end{equation}

If additionally there is $\psi_1\in\Kinf$ so that \eqref{LyapFunk_1Eig_LISS} holds, 
then $V$ is a \emph{coercive ISS Lyapunov function} for $\Sigma$.
\end{lemma}

\begin{proof}
$\Leftarrow$. Follows from the assumption that $\inf_{\tau > 0}\|u_\tau\|_\Uc \leq \|u\|_\Uc$ for all $u\in\Uc$.

$\Rightarrow$. Let a continuous function $V:X \to \R_+$ be a non-coercive ISS Lyapunov function.
Pick any $x \in X$ and any $u\in\Uc$.
Since $\Uc$ is a linear space, $0\in\Uc$ and the axiom of concatenation implies that $u_\tau \in \Uc$ for any $\tau > 0$.
By definition of $\dot{V}_u$, for any $\tau>0$ it holds that $\dot{V}_u(x) = \dot{V}_{u_\tau}(x)$, and thus
\begin{equation*}
\dot{V}_u(x) = \dot{V}_{u_\tau}(x) \leq -\alpha(\|x\|_X)  + \sigma(\|u_\tau\|_{\Uc}).
\end{equation*}
Taking infimum over $\tau>0$, we obtain \eqref{DissipationIneq_nc-inf}.
\end{proof}

Our definition of an ISS Lyapunov function is defined for any normed linear space $\Uc$, that allows to develop the ISS Lyapunov theory for a very broad class of systems. However, for some input spaces this definition is far too restrictive, and for other systems simpler and more useful restatements of the ISS Lyapunov function concept may be more useful.
The most important input spaces we consider next.

\begin{remark}[$L^p$ input spaces]
\label{rem:Lyapunov-functions-for-Lp-spaces} 
For spaces $\Uc=L^p(\R_+,U)$, for a Banach space $U$ and some $p\in[1,+\infty)$, Definition~\ref{def:noncoercive_ISS_LF}
 is far too restrictive.
Indeed, for any $u\in\Uc = L^p(\R_+,U)$ it holds that 
$\inf_{\tau\ge 0}\|u_\tau\|_{\Uc} = 0$, and thus the inequality \eqref{DissipationIneq_nc-inf}  reduces to
\begin{equation}
\label{DissipationIneq_nc-Lp}
\dot{V}_u(x) \leq -\alpha(\|x\|_X),
\end{equation}
which ensures a much stronger property than ISS, namely the uniform global asymptotic stability (for a precise definition see, e.g., \cite{MiW19a}).

Coercive and non-coercive ISS Lyapunov theory which is appropriate for systems with $\Uc=L^p(\R_+,U)$, $p\in[1,+\infty)$, has been developed in \cite{Mir20}.
\end{remark}

\begin{remark}[Piecewise continuous input functions]
If $\Uc=PC_b(\R_+,U)$, for a Banach space $U$, then the dissipation inequality \eqref{DissipationIneq_nc-inf} simplifies to 
\begin{equation}
\label{DissipationIneq_nc-PC_b}
\dot{V}_u(x) \leq -\alpha(\|x\|_X)  + \sigma(\|u(0)\|_{U}),
\end{equation}
which resembles the classical dissipation inequality, used in the ISS theory of ODE systems.
Similarly to Proposition~\ref{prop:ncLF_implies_norm-to-integral-ISS}, one can show that 
for a forward-complete system $\Sigma=(X, PC_b(\R_+,U), \phi)$, the existence of a non-coercive ISS Lyapunov function implies integral-to-integral ISS.
\end{remark}

In the rest of this section, let $\Uc:=L^\infty(\R_+,U)$. In this case our definition of ISS Lyapunov function seems to be the most natural. 
For forward complete finite-dimensional systems with $f$ as in the statement of Corollary~\ref{cor:ni-ISS-and-ISS-ODEs}, existence of a non-coercive ISS Lyapunov function implies not only norm-to-integral ISS, but also integral-to-integral ISS, which follows from Proposition~\ref{prop:ncLF_implies_norm-to-integral-ISS} and 
 Corollary~\ref{cor:ni-ISS-and-ISS-ODEs}.
However, for infinite-dimensional systems we were able to show only somewhat weaker property (which is still stronger than norm-to-integral ISS, but weaker than integral-to-integral ISS), see \eqref{eq:old-int-to-int-ISS}. Thus, a question remains whether existence of an ISS Lyapunov function (coercive or non-coercive) as defined in Definition~\ref{def:noncoercive_ISS_LF} implies integral-to-integral ISS for forward complete systems.
Although we do not have an answer to this problem, in the following proposition we show that if an ISS Lyapunov function satisfies a somewhat stronger dissipative estimate, then integral-to-integral ISS can be verified.

\begin{proposition}
\label{prop:ncLF_implies_ii_ISS} 
Let $\Sigma=(X,\Uc,\phi)$ be a forward complete control system with the input space $\Uc:=L^\infty(\R_+,U)$. 
Assume that there is a continuous function $V:X \to \R_+$, $\psi_2,\alpha \in \Kinf$ and $\sigma \in \K$ such that
\eqref{LyapFunk_1Eig_nc_ISS} holds and the Dini derivative of $V$ along the trajectories of $\Sigma$ for all $x \in X$ and $u\in \Uc$ satisfies
\begin{equation}
\label{DissipationIneq_nc-Linfty-stronger-form}
\dot{V}_u(x) \leq -\alpha(\|x\|_X)  + \Big[D_{+,\tau}\Big(\int_0^\tau\sigma(\|u(s)\|_{U})ds\Big)\Big]_{\tau=0},
\end{equation}
where $D_{+,\tau}$ means that lower right-hand Dini derivative is taken with respect to the argument $\tau$.

Then $\Sigma$ is integral-to-integral ISS.
\end{proposition}

Before we prove this proposition, note that the estimate \eqref{DissipationIneq_nc-Linfty-stronger-form} implies \eqref{DissipationIneq_nc}, since for any $x \in X$ and $u\in \Uc$ it holds that 
\begin{eqnarray*}
\Big[D_{+,\tau}\Big(\int_0^\tau\sigma(\|u(s)\|_{U})ds\Big)\Big]_{\tau=0}
\leq
\Big[D_{+,\tau}\Big(\int_0^\tau\sigma(\|u\|_{\Uc})ds\Big)\Big]_{\tau=0}
= \sigma(\|u\|_{\Uc}).
\end{eqnarray*} 
Thus, function $V$ as in Proposition~\ref{prop:ncLF_implies_ii_ISS} is an ISS Lyapunov function for $\Sigma$, with a (potentially) stronger dissipative estimate.

\begin{proof}
Assume that $V$ is a non-coercive ISS Lyapunov function for $\Sigma$ with
the corresponding $\psi_2,\alpha,\sigma$.
Pick any $u\in\Uc$ and any $x\in X$. As we assume forward completeness of $\Sigma$, the trajectory $\phi(\cdot,x,u)$ exists for all $t\geq0$ and due to \eqref{DissipationIneq_nc}, we have for any $t>0$ that:
\begin{equation}
\label{DissipationIneq_nc-specified-i-to-i-ISS}
\dot{V}_{u(t+\cdot)}\big(\phi(t,x,u)\big) \leq -\alpha(\|\phi(t,x,u)\|_X) + \Big[D_{+,\tau}\Big(\int_0^\tau\sigma(\|u(t+s)\|_{U})ds\Big)\Big]_{\tau=0}.
\end{equation}
The last term can be rewritten in a simpler form:
\begin{align*}
\Big[D_{+,\tau}\Big(\int_0^\tau\sigma(\|u(t&+s)\|_{U})ds\Big)\Big]_{\tau=0}
=\Big[D_{+,\tau}\Big(\int_t^{t+\tau}\sigma(\|u(s)\|_{U})ds\Big)\Big]_{\tau=0}\\
&=\Big[D_{+,\tau}\Big(\int_0^{t+\tau}\sigma(\|u(s)\|_{U})ds\Big)\Big]_{\tau=0}
=D_{+,t}\Big(\int_0^{t}\sigma(\|u(s)\|_{U})ds\Big).
\end{align*}
Arguing as in Proposition~\ref{prop:ncLF_implies_norm-to-integral-ISS}, we define 
$y(t):=V\big(\phi(t,x,u)\big)$ and verify that the equality  
$\dot{V}_{u(t+\cdot)}\big(\phi(t,x,u)\big) = D^+ y(t)$ 
holds and $y(0)= V(x)$.

%

In view of the continuity axiom of $\Sigma$, for fixed $x,u$ the map $\phi(\cdot,x,u)$ is continuous, and thus $t\mapsto -\alpha(\|\phi(t,x,u)\|_X)$ is continuous as well. Hence, we can rewrite the inequality \eqref{DissipationIneq_nc-specified-i-to-i-ISS} as
\begin{eqnarray*}
D^+y(t) \leq - \frac{d}{dt}\int_0^t \alpha(\|\phi(s,x,u)\|_X) ds + D_+\Big(\int_0^t\sigma(\|u(s)\|_{U})ds\Big). 
\end{eqnarray*}
As $-D_+\big(\int_0^t\sigma(\|u(s)\|_{U})ds\big) = D^+\big(-\int_0^t\sigma(\|u(s)\|_{U})ds\big)$, we proceed to
\begin{eqnarray*}
D^+y(t) + \frac{d}{dt}\int_0^t \alpha(\|\phi(s,x,u)\|_X) ds + D^+ \Big(-\int_0^t\sigma(\|u(s)\|_{U})ds\Big) \leq 0.
\end{eqnarray*}
As in the proof of Proposition~\ref{prop:ncLF_implies_norm-to-integral-ISS}, by subadditivity of $D^+$, we obtain that
\begin{eqnarray*}
D^+\Big(y(t) + \int_0^t \alpha(\|\phi(s,x,u)\|_X) ds -\int_0^t\sigma(\|u(s)\|_{U})ds\Big) \leq 0.
\end{eqnarray*}
It follows from 
\cite[Theorem 2.1]{Sza65} that $t\mapsto y(t) + \int_0^t \alpha(\|\phi(s,x,u)\|_X) ds -\int_0^t\sigma(\|u(s)\|_{U})ds$ 
is nonincreasing. As $G(0) = 0$, for all $t\ge0$ we have
\begin{eqnarray*}
y(t) + G(t) -\int_0^t\sigma(\|u(s)\|_{U})ds \leq y(0) = V(x)\leq  \psi_2(\|x\|_X). 
\end{eqnarray*}
Now $y(t)\geq 0$ for all $t\in\R_+$, and so
\begin{align*}
\int_0^t \alpha(\|\phi(s,x,u)\|_X) ds 
\leq \psi_2(\|x\|_X) + \int_0^t\sigma(\|u(s)\|_{U})ds.
\end{align*}
This completes the proof.
\end{proof}

\section{Construction of ISS Lyapunov functions for linear systems with unbounded input operators}
\label{sec:Lyapunov_Theorem_systems_unbounded_operators}

In the remainder of this paper, we specialize to the case of complex Hilbert spaces $X$ endowed with the scalar 
product $\scalp{\cdot}{\cdot}_X$ and to the input spaces
$\Uc:=L^\infty(\R_+,U)$, where $U$ is a Banach space. 
We denote the norm in $\Uc$ by $\|\cdot\|_\infty$.

A classical method for construction of Lyapunov functions for
exponentially stable semigroups is the solution of the operator Lyapunov
equation \cite[Theorem 5.1.3]{CuZ95}. This method can also be used for the construction of ISS Lyapunov functions for systems with bounded input operators \cite{MiW17c}.

The following result holds (this is a Hilbert-space version of \cite[Proposition 6]{MiW17c}, and thus we omit the proof):
\begin{proposition}
\label{prop:Converse_ISS_Lyapunov_theorem_lin_Systems_with_bounded_input_operators} 
Let $X$ be a complex Hilbert space, $\Uc:=L^\infty(\R_+,U)$, where $U$ is a Banach space, $A$ generate a strongly continuous semigroup over $X$ and let $B\in L(U,X)$.
If \eqref{InfiniteDim2} is ISS, then there is an operator $P=P^* \in L(X)$ so that $\scalp{Px}{x}_X>0$ for $x\neq 0$
and $P$ solves the Lyapunov equation
\begin{eqnarray}
\scalp{Px}{Ax}_X +\scalp{Ax}{Px}_X = -\|x\|^2_X,\quad x\in D(A).
\label{eq:Gn_Lyap_Equation}
\end{eqnarray}
Furthermore, $V:X\to\R_+$ defined by
\begin{align}
V(x) =\scalp{Px}{x}_X
\label{eq:QuadraticLF}
\end{align}
is a non-coercive  ISS Lyapunov function for \eqref{InfiniteDim2}.
\end{proposition}

\begin{remark}
\label{rem:Sense-of-Proposition-nc-LF} 
It is well-known that the existence of a positive solution to the Lyapunov equation \eqref{eq:Gn_Lyap_Equation}
is equivalent to the exponential stability of the semigroup generated by $A$, see \cite[Theorem 5.1.3]{CuZ95}, and thus to the ISS of the system \eqref{InfiniteDim2} (for admissible $B$). The main contribution of Proposition~\ref{prop:Converse_ISS_Lyapunov_theorem_lin_Systems_with_bounded_input_operators} is that the quadratic function $V$, defined by \eqref{eq:QuadraticLF}, is also a non-coercive ISS Lyapunov function  for \eqref{InfiniteDim2}.
\end{remark}

Note that the operator $P$ and the Lyapunov function $V$ can be chosen independently on the bounded input operator $B$ (the function $\sigma$ in the dissipative estimate \eqref{DissipationIneq_nc} however does depend on $B$). In the next theorem we derive a counterpart of 
Proposition~\ref{prop:Converse_ISS_Lyapunov_theorem_lin_Systems_with_bounded_input_operators} 
for systems with merely admissible operators $B$.
In contrast to systems with bounded input operators, we need further assumptions that relate the operators $P$ and $A$.

\begin{theorem}
\label{thm:Gen_ISS_LF_Construction}
Let $A$ be  the generator of a $C_0$-semigroup $(T(t))_{t\ge 0}$ on a complex Hilbert space $X$ 
 and let $\Uc:=L^\infty(\R_+,U)$, where $U$ is a Banach space.

Assume that there is an operator $P\in L(X)$ satisfying the following conditions:
\begin{itemize}
	\item[(i)] \label{item:Gen_Converse_Lyap_Theorem_1} $P$ satisfies
\begin{eqnarray}
\re\scalp{Px}{x}_X > 0,\quad x\in X\backslash\{0\}.
\label{eq:Positivity}
\end{eqnarray}
\item[(ii)] \label{item:Gen_Converse_Lyap_Theorem_3}  $\im(P) \subset D(A^*)$.
\item[(iii)] \label{item:Gen_Converse_Lyap_Theorem_4} $PA$ has an
  extension to a bounded operator on $X$, that is, $PA\in L(X)$. (We
  also denote this
  extension by $PA$.)
\item[(iv)] \label{item:Gen_Converse_Lyap_Theorem_2} $P$ satisfies the Lyapunov inequality 
\begin{eqnarray}
\re\scalp{(PA+A^*P)x}{x}_X \leq  -\scalp{x}{x}_X,\quad x\in D(A),
\label{eq:LyapIneq}
\end{eqnarray}
\end{itemize}
Then for any $\infty$-admissible input operator $B\in L(U,X_{-1})$ the function
\begin{equation}
\label{Lyap}
 V(x) := \re\, \langle Px,x\rangle_X
 \end{equation}
is a non-coercive ISS Lyapunov function for \eqref{InfiniteDim2}, which
satisfies for each $\eps>0$ the dissipation inequality
\begin{align}
\label{ex:Dissipative_Inequality_linear_system}
 \dot{V}_u(x_0) \le  (\varepsilon-1) \|x_0\|_X^2 + c(\eps) \|u\|^2_\infty,
\end{align}
where 
\begin{align}
\label{eq:ceps-def}
c(\eps):=&\ \frac{1}{4\varepsilon}\big(\|A^\ast P\|_{L(X)} +\|PA\|_{L(X)}\big)^2 \|A^{-1}_{-1}B\|_{L(U,X)}^2M^2\\
&\qquad\qquad\qquad+ M\|A^\ast P\|_{L(X)} \| A_{-1}^{-1}B\|_{L(U,X)}
\kappa(0).
\nonumber
\end{align}
and
$\kappa(0)=\lim_{t \searrow 0}\kappa(t)$, where  $\kappa(t)>0$ is the smallest constant satisfying
\begin{equation}\label{kappa} \left\|\int_0^t T_{-1}(t-s)Bu(s)\, ds\right\|_X \le \kappa(t) \|u\|_{\infty}, \end{equation}
for every $u\in L^\infty([0,t),U)$. (The existence of the constants $\kappa(t)$ is implied by the $\infty$-admissibility of $B$.)

In particular, existence of a non-coercive ISS Lyapunov function \eqref{Lyap} implies that \eqref{InfiniteDim2} is ISS for any $\infty$-admissible $B$.
\end{theorem}

\begin{remark}
\label{rem:Why-Real-parts} 
Note that we have to take the real parts of the expressions in \eqref{Lyap} and \eqref{eq:LyapIneq}, as we deal with complex Hilbert spaces and we do not assume that $P$ is a positive operator on $X$.
\end{remark}

\begin{remark}
\label{rem:Classical_Form_Lyapunov_Equation} 
If in addition to the assumptions of Theorem
\ref{thm:Gen_ISS_LF_Construction} the operator $P$ is self-adjoint, that
is, if $P=P^*$, then equation \eqref{eq:Gn_Lyap_Equation} is equivalent to \eqref{eq:LyapIneq}. 

\end{remark}

\begin{proof}
Note that linear systems with admissible input operators satisfy both the
CEP and the BRS property, which follows easily from inequality \eqref{eq:Bounds-on-convolution}.
Due to Theorem~\ref{t:ISSLyapunovtheorem}, $\Sigma$ is ISS if $V$ is a non-coercive ISS Lyapunov function for $\Sigma$.

By the assumptions
\[
0 < V(x) \leq \|P\|_{L(X)}\|x\|^2_X, \quad x\in X\backslash \{0\},
\]
and thus \eqref{LyapFunk_1Eig_nc_ISS} holds.
It remains to show the dissipation inequality \eqref{DissipationIneq_nc} for $V$.

The operator $A:D(A)\subset X\rightarrow X$ is densely defined as an infinitesimal generator of a $C_0$-semigroup, and hence $A^*$ is well-defined and again the generator of a $C_0$-semigroup, see \cite[Corollary 10.6]{Paz83}. In particular, this implies that $A^*$ is a closed operator. 
Since $P \in L(X)$, the operator $S:=A^*P$ with the domain $D(S):=\{x\in X: Px \in D(A^*)\}$ is a closed operator, see \cite[Exercise 5.6]{Wei80}.
However, by our assumptions $\im(P) \subset D(A^*)$, which implies $D(S) = X$,
and thus $S = A^*P \in L(X)$ by the Closed Graph Theorem. In particular, the term $\|A^\ast P\|_{L(X)}$ in 
\eqref{eq:ceps-def} makes sense.


For $x_0\in X$ and $u\in L^\infty([0,\infty),U)$, we have
{\allowdisplaybreaks
\begin{align}
V(\phi(t,x_0,u))&-V(x_0) \label{eq:MainEquality}\\
=&\, \re\, \left\langle P\left(T(t)x_0+\int_0^t T_{-1}(t-s)Bu(s)ds\right),\right.\nonumber\\
&\qquad  \left. T(t)x_0+\int_0^t T_{-1}(t-s)Bu(s)ds\right\rangle_X -\re\,\langle Px_0,x_0\rangle_X\nonumber\\
=&\, \re\,\langle PT(t)x_0,x_0\rangle_X -\re\,\langle Px_0,x_0\rangle_X \nonumber\\
&+ \re\,\langle PT(t)x_0,T(t)x_0\rangle_X - \re\,\langle PT(t)x_0,x_0\rangle_X \label{eq:line2}\\
&+ \re\, \Big\langle PT(t)x_0 ,\int_0^t T_{-1}(t-s)Bu(s)ds\Big\rangle_X \label{eq:line3}\\
&\qquad + \re\,\left\langle P\int_0^t T_{-1}(t-s)Bu(s)ds ,T(t)x_0\right\rangle_X  \label{eq:line4}\\
&+ \re\, \left\langle P\int_0^t T_{-1}(t-s)Bu(s)ds ,\int_0^t T_{-1}(t-s)Bu(s)ds\right\rangle_X.   \label{eq:line5}
\end{align}
}
The terms in line \eqref{eq:line2} of the previous expression can be transformed into:
\begin{align}
\re\,\langle PT(t)x_0,T(t)&x_0\rangle_X - \re\,\langle PT(t)x_0,x_0\rangle_X = \re\,\langle PT(t)x_0,T(t)x_0 - x_0\rangle_X. \label{eq:line2_Estimate}
\end{align}
Applying \cite[Theorem 5.1.3]{CuZ95} to the operator $\tfrac{1}{2}(P+P^*)$, we see that the conditions
 (i), (ii) and (iv) imply that $A$ generates an exponentially stable
 semigroup. 
This implies (see e.g. \cite[Proposition 5.2.4]{JaZ12}) that $0\in\rho(A)$ and thus $A^{-1}\in L(X)$  exists. Further, the exponential stability of the $C_0$-semigroup $(T(t))_{t\ge 0}$ implies 
\[ \|T(t)\|_{L(X)} \le Me^{-\omega t},\quad t\ge 0,\]
for some constants $M$, $\omega>0$. Thanks to $\rho(A)=\rho(A_{-1})$, the operator $A_{-1}^{-1}$ exists  as well. 

By \cite[Theorem II.5.5]{EnN00}, the map $A:D(A) \to X$ can be continuously extended to the linear isometry
 $A_{-1}$ which maps $(X,\|\cdot\|_X)$ onto $(X_{-1},\|\cdot\|_{X_{-1}})$.
Hence $A_{-1}^{-1}$,
mapping $(X_{-1},\|\cdot\|_{X_{-1}})$ onto $(X,\|\cdot\|_X)$, is again a
  linear isometry, and thus a bounded operator.
As $B\in L(U,X_{-1})$, we have that $A^{-1}_{-1}B \in L(U, X)$. In particular, 
$T_{-1}(t-s)A_{-1}^{-1}Bu(s) = T(t-s)A_{-1}^{-1}Bu(s) \in X$ for a.e. $s\geq 0$.
Due to the fact, that $A_{-1}^{-1}$ and $T_{-1}(t-s)$ commute, we obtain
{\allowdisplaybreaks
\begin{align}
\Big\| A_{-1}^{-1} \int_0^t T_{-1}(t-s)Bu(s)ds \Big\|_X
&=   \Big\| \int_0^t A_{-1}^{-1}T_{-1}(t-s)Bu(s)ds \Big\|_X\nonumber\\
&=  \Big\| \int_0^t T_{-1}(t-s)A_{-1}^{-1}Bu(s)ds \Big\|_X.\nonumber\\
&\leq    \int_0^t \|T(t-s)\|_{L(X)}\|A_{-1}^{-1}B\|_{L(U,X)}\|u(s)\|_U ds \nonumber\\
&\leq   \int_0^t M ds \|A_{-1}^{-1}B\|_{L(U,X)}\|u\|_\infty \nonumber\\
&\leq   M t \|A_{-1}^{-1}B\|_{L(U,X)}\|u\|_\infty.  \label{eq:est}
\end{align}}

Since $\im(P)\subset D(A^*)$,  we estimate the expression in
\eqref{eq:line3} using the Cauchy-Schwarz inequality and \eqref{eq:est}
{\allowdisplaybreaks
\begin{align}
 \re\, \Big\langle P T(t)x_0 ,\int_0^t T_{-1}(t-s)&Bu(s)ds\Big\rangle_X\nonumber\\
 =&\re\, \Big\langle P T(t)x_0 ,A A^{-1}\int_0^t  T_{-1}(t-s)Bu(s)ds\Big\rangle_X\nonumber\\
=&\re\,\Big\langle  A^\ast P T(t)x_0 , A_{-1}^{-1} \int_0^t T_{-1}(t-s)Bu(s)ds\Big\rangle_X\nonumber\\
\leq & \|A^\ast P T(t)x_0\|_X \cdot \Big\| A_{-1}^{-1} \int_0^t T_{-1}(t-s)Bu(s)ds \Big\|_X\nonumber\\
%
%
\leq & \|A^\ast P\|_{L(X)} \|T(t)x_0\|_X \cdot  M t \|A_{-1}^{-1}B\|_{L(U,X)}\|u\|_\infty.  \label{eq:line3_Estimate}
\end{align}
}
To obtain an upper bound for the expression \eqref{eq:line4}, we use again \eqref{eq:est} to obtain 
\begin{align}
\re\,\Big\langle P\int_0^t T_{-1}(t-s)Bu(s)ds,& T(t)x_0\Big\rangle_X\nonumber\\
=	& \re\,\left\langle PAA^{-1}\int_0^t T_{-1}(t-s)Bu(s)ds ,T(t)x_0\right\rangle_X\nonumber\\
\leq & \|PA\|_{L(X)} \cdot  M t \|A_{-1}^{-1}B\|_{L(U,X)}\|u\|_\infty \|T(t)x_0\|_X.  \label{eq:line4_Estimate_b}
\end{align}
Finally, we estimate the expression \eqref{eq:line5} using \eqref{eq:est} and $\kappa$ as defined in \eqref{kappa}
\begin{align}
\re\,\Big\langle P\int_0^t T_{-1}&(t-s)Bu(s)ds ,\int_0^t T_{-1}(t-s)Bu(s)ds\Big\rangle_X\nonumber\\
&\! =  \re\,\left\langle P\int_0^t \!T_{-1}(t-s)Bu(s)ds ,AA^{-1}_{-1}\int_0^t \!T_{-1}(t-s)Bu(s)ds\!\!\right\rangle_X\nonumber\\
&\! =  \re\,\left\langle A^*P\!\int_0^t \!T_{-1}(t-s)Bu(s)ds ,A^{-1}_{-1}\!\int_0^t \!T_{-1}(t-s)Bu(s)ds\!\!\right\rangle_X\nonumber\\
&\!\leq  \|A^*P\|_{L(X)} \kappa(t)\|u\|_\infty \cdot Mt \|A_{-1}^{-1}B\|_{L(U,X)}\|u\|_\infty.   \label{eq:line5_Estimate}
\end{align}
Substituting \eqref{eq:line2_Estimate}, \eqref{eq:line3_Estimate}, 
\eqref{eq:line4_Estimate_b}, and \eqref{eq:line5_Estimate} into \eqref{eq:MainEquality}, we obtain:
\begin{align*}
V(\phi(t,x_0,u))-V(x_0) \le \,& \re\, \langle PT(t)x_0 - Px_0, x_0\rangle_X + \re\,\langle PT(t)x_0,T(t)x_0 - x_0\rangle_X\\
&+\|A^\ast P\|_{L(X)} \|T(t)x_0\|_X \cdot  M t \|A_{-1}^{-1}B\|_{L(U,X)}\|u\|_\infty\\
&+ \|PA\|_{L(X)} \cdot  M t \|A_{-1}^{-1}B\|_{L(U,X)}\|u\|_\infty \|T(t)x_0\|_X\\
& + \|A^*P\|_{L(X)} \kappa(t)\|u\|_\infty \cdot Mt \|A_{-1}^{-1}B\|_{L(U,X)}\|u\|_\infty.
\end{align*}
For $x_0\in X$, we have $A^{-1}x_0\in D(A)$ and we obtain from the definition of the generator $A$ that
\begin{align*}
\mathop{\overline{\lim}}_{t\searrow 0} \re\, \frac{1}{t}\langle PT(t)x_0 - Px_0, x_0\rangle_X
& =\mathop{\overline{\lim}}_{t\searrow 0} \re\, \frac{1}{t}\langle PA[ T(t)A^{-1}x_0 - A^{-1}x_0], x_0\rangle_X\\
&= \re\scalp{PAx_0}{x_0}_X
\end{align*}
and similarly
\begin{align*}
\mathop{\overline{\lim}}_{t\searrow 0} \re\, \frac{1}{t}\langle PT(t)x_0,T(t)x_0 - x_0\rangle_X 
& =\mathop{\overline{\lim}}_{t\searrow 0} \re\, \frac{1}{t}\langle A^*PT(t)x_0 , T(t)A^{-1}x_0 - A^{-1}x_0\rangle_X\\
&= \re\scalp{A^* Px_0}{x_0}_X.
\end{align*}
This implies for every $\varepsilon>0$ that (recall the definition of $\kappa(0)$ before \eqref{kappa})
{\allowdisplaybreaks
\begin{align*}
 \dot{V}_u(x_0) =\mathop{\overline{\lim}}_{t\searrow 0} \frac{1}{t}&\Big(V\big(\phi(t,x_0,u)\big)-V(x_0)\Big)\\
\le &\  \re\scalp{PAx_0}{x_0}_X + \re\,\scalp{A^*Px_0}{x_0}_X \\
&  + \|A^\ast P\|_{L(X)} \|x_0\|_X  \|A_{-1}^{-1}B\|_{L(U,X)} M  \|u\|_\infty\\
&+  \|PA\|_{L(X)}\|  A_{-1}^{-1}B\|_{L(U,X)}  M  \|u\|_\infty  \|x_0\|_X\\
& +   \|A^\ast P\|_{L(X)} \|  A_{-1}^{-1}B\|_{L(U,X)}  M  \kappa(0)\|u\|^2_\infty\\
=&\    \re\scalp{PAx_0}{x_0}_X + \re\,\scalp{A^*Px_0}{x_0}_X \\
&  +  \| x_0\|_X (\|A^\ast P\|_{L(X)} + \|PA\|_{L(X)})  \|A_{-1}^{-1}B\|_{L(U,X)} M  \|u\|_\infty\\
& +  \|A^\ast P\|_{L(X)} \|  A_{-1}^{-1}B\|_{L(U,X)}  M  \kappa(0)\|u\|^2_\infty.
\end{align*}
Using Young's inequality and the estimate \eqref{eq:LyapIneq} we proceed to
\begin{align*}
 \dot{V}_u(x_0) \le&  - \|x_0\|_X^2 + \varepsilon  \|x_0\|_X^2  + \frac{(\|A^\ast P\|_{L(X)} +\|PA\|_{L(X)})^2 \|A^{-1}_{-1} B\|_{L(U,X)}^2M^2 }{4\varepsilon} \|u\|^2_\infty  \\
& + \|A^\ast P\|_{L(X)}\|  A_{-1}^{-1}B\|_{L(U,X)}  M  \kappa(0)\|u\|^2_\infty,
\end{align*}
which shows the dissipation inequality \eqref{ex:Dissipative_Inequality_linear_system}, and thus also
\eqref{DissipationIneq_nc}.
}
\end{proof}

\begin{remark}
\label{rem:Reformulation_of_Linear_Theorem} 
Theorem~\ref{thm:Gen_ISS_LF_Construction} has been formulated as a direct
Lyapunov theorem. However, the following reformulation as a partial converse result is also possible. Assume that \eqref{InfiniteDim2} is ISS, and the solution $P$ of the Lyapunov equation \eqref{eq:Gn_Lyap_Equation} satisfies $\im(P) \subset D(A^*)$ and $PA$ is bounded. Then \eqref{Lyap} is an ISS Lyapunov 
function for \eqref{InfiniteDim2}. 
\end{remark}

It is of virtue to compare the ISS Lyapunov theorem for bounded input operators (Proposition~\ref{prop:Converse_ISS_Lyapunov_theorem_lin_Systems_with_bounded_input_operators}) and ISS Lyapunov theorem for admissible input operators (Theorem~\ref{thm:Gen_ISS_LF_Construction}). The ISS Lyapunov function candidate considered in both these results, is the same. What differs is the assumptions and the set of input operators, for which this function is indeed an ISS Lyapunov function. Proposition~\ref{prop:Converse_ISS_Lyapunov_theorem_lin_Systems_with_bounded_input_operators} states that if the semigroup, generated by $A$ is exponentially stable, then there  is an operator $P$, which satisfies the assumptions (i) and (iv) of Theorem~\ref{thm:Gen_ISS_LF_Construction} and the condition $P=P^*$, and furthermore \eqref{eq:QuadraticLF} is an ISS Lyapunov function for \eqref{InfiniteDim2} for any bounded input operator. 
\emph{Thus, the key additional assumptions which we impose in order to tackle the unboundedness of an input operator, are the assumptions (ii) and (iii).}
We note, that  that with these assumptions \eqref{Lyap} is an ISS Lyapunov function for any $\infty$-admissible operator $B$.

%
%
%
%
%

\section{Applications of Theorem~\ref{thm:Gen_ISS_LF_Construction}}
\label{sec:Applications}

In this section, we show applicability of Theorem~\ref{thm:Gen_ISS_LF_Construction} for some important special cases.
We start with sufficient conditions, guaranteeing that Theorem~\ref{thm:Gen_ISS_LF_Construction} can be applied with $P=-A^{-1}$.
Then we show that these sufficient conditions are fulfilled for broad classes of systems, generated by subnormal operators. Finally, we proceed to diagonal semigroups (whose generators are self-adjoint operators) and finally we give a construction of a non-coercive ISS Lyapunov function for a heat equation with Dirichlet boundary inputs.

\subsection{A special case: $P=-A^{-1}$}
\label{sec:choosing_A_inverse}

In this section we give sufficient conditions for the applicability of Theorem~\ref{thm:Gen_ISS_LF_Construction} with $P:=-A^{-1}$.

\begin{proposition}\label{prop:Conv_ISS_LF_Theorem_LinOp}
Let $A$ be  the generator of an exponentially stable $C_0$-semigroup $(T(t))_{t\ge 0}$ on a (complex) Hilbert space $X$
 and let $\Uc:=L^\infty(\R_+,U)$, where $U$ is a Banach space.

Further, assume that
\begin{itemize}
	\item[(a)] $D(A)\subset D(A^\ast)$
	\item[(b)] there is $\delta\in(0,1)$ such that for every $x \in X$ we have 
\begin{equation}\label{eqn:a2}
 \re\, \langle A^*A^{-1}x, x\rangle_X + \delta \|x\|_X^2\ge 0
\end{equation}
	\item[(c)] $  \re\,\langle Ax,x\rangle_X <0$ holds for  every $x\in D(A)\backslash\{0\}$.
\end{itemize}
Then
\begin{equation}\label{Lyap-Apower--1}
 V(x) := - \re\, \langle A^{-1}x,x\rangle_X
 \end{equation}
is an ISS Lyapunov function for  \eqref{InfiniteDim2} for any $\infty$-admissible operator $B\in L(U,X_{-1})$.
\end{proposition}

\begin{proof}
As $A$ generates an exponentially stable semigroup, $0 \in\rho(A)$ and thus $P:=-A^{-1} \in L(X)$. We show step by step 
that this choice of $P$ satisfies all the requirements (i)--(iv) of Theorem~\ref{thm:Gen_ISS_LF_Construction}.\\
\textbf{(i).} For any $x\in X\backslash\{0\}$ there is $y \in D(A)\backslash\{0\}$ so that $x = Ay$. Then by the assumptions of the proposition it holds that
\[
V(x) = -\re\scalp{y}{Ay}_X = -\re\scalp{Ay}{y}_X >0.
\]
\textbf{(ii).} We have $\im(P)=\im(A^{-1}) = D(A) \subset D(A^*)$, which holds by our assumptions.\\
\textbf{(iii).} Trivial as $PA = -I$.\\
\textbf{(iv).} 
By assumptions, there is a $\delta<1$ so that
\begin{eqnarray*}
\re\scalp{(PA+A^*P)x}{x}_X = \re\scalp{(-I-A^*A^{-1})x}{x}_X
												&=& -\scalp{x}{x}_X - \re\scalp{A^*A^{-1} x}{x}_X\\
												&\leq& -(1-\delta)\scalp{x}{x}_X,
\end{eqnarray*}
and thus $P$ satisfies the Lyapunov inequality up to a scaling coefficient (and $\tilde{P}:=\frac{1}{1-\delta}P$ satisfies precisely \eqref{eq:LyapIneq}).

Hence all assumptions of Theorem~\ref{thm:Gen_ISS_LF_Construction} are satisfied, and an application of 
Theorem~\ref{thm:Gen_ISS_LF_Construction} shows the claim.
\end{proof}

\begin{remark}
\label{rem:Equivalent_inequality}
If $D(A)\subset D(A^\ast)$, then
inequality \eqref{eqn:a2} is equivalent to the existence of a constant $\delta'<1$ satisfying
\[ \|(A+A^\ast)x\|_X^2 + \delta' \|Ax\|_X^2 \ge \|A^\ast x\|_X^2,\qquad x\in D(A).\]
If $A$ generates a strongly continuous contraction semigroup, then \eqref{eqn:a2} implies that the semigroup $(T(t))_{t\ge 0}$ is 2-hypercontractive \cite{JPP17}. In particular, subnormal and normal operators whose spectrum lies in a sector, satisfy  \eqref{eqn:a2}, see Proposition \ref{propsub}.
\end{remark}

\subsection{Analytic semigroups generated by subnormal operators}
\label{sec1}


In this section we show that Theorem~\ref{thm:Gen_ISS_LF_Construction} can be applied to a broad class of analytic semigroups over Hilbert spaces generated by subnormal operators.

A closed, densely-defined operator $A$ on a Hilbert space $X$ is called \emph{subnormal}, if there is a Hilbert space $Z$ containing $X$ as a subspace and a normal operator $(N,D(N)):Z\to Z$ so that $A=N_{|X}$ (the restriction of $N$ to $X$)
and $X$ is an invariant subspace for $N$, that is, $N(D(N)\cap X)\subset X$.
We write $P$ for the
orthogonal projection from $Z$ onto $X$.

By \cite[Th. X.4.19]{Con90}, there is a measure space $(W, \Omega, \mu)$ and an $\Omega$-measurable function 
$\nu$ such that $N$ is unitarily equivalent to $M_\nu$, a multiplication operator on $L^2(\mu):=L^2(W,\Omega,\mu)$.
Note that in \cite[Th. X.4.19]{Con90} it was assumed that the Hilbert space $Z$ is separable, but the result can be extended to the non-separable case.
Furthermore, $Ax=\nu x$, $A^* x=P(\overline \nu x)$.
If $Z$ is a separable Hilbert space, then the measure space $(W, \Omega, \mu)$ is $\sigma$-finite. Note also that separability of $X$ ensures separability of $Z$. Indeed, if an operator on a separable Hilbert space has a normal extension, then its minimal normal extension will certainly act on a
separable space, see \cite{Con91} as a general reference.



We denote the spectrum of a linear closed operator $A$ by $\sigma(A)$. It holds that:
\begin{lemma}
\label{lem:Domains_of_subnormal_operators} 
A  subnormal operator $A$ satisfies $D(A)\subset D(A^*)$. Further, there exists a minimal normal extension $N$ satisfying $\sigma(N)\subset\sigma(A)$.
\end{lemma}

\begin{proof}
The first assertion follows  from $D(N) = D(N^*)$, see \cite[Prop. X.4.3]{Con90}. The second assertion is proved in \cite[Theorem 2.3]{PaBh89}.
\end{proof}

\begin{example}
\begin{enumerate}
\item Clearly, every normal operator on a Hilbert space is subnormal.
\item Symmetric operators on Hilbert spaces and analytic Toeplitz operators $T_g$ on the Hardy space $\text{\bf H}^2(\D)$ are subnormal,  \cite{PaBh89}.
\end{enumerate}
\end{example}
For $\theta \in [0,\pi/2)$ we define
\[ S_\theta := \{s\in \mathbb C\mid |{\rm arg}\,(- s)|\le \theta \}.\]

\begin{proposition}\label{prop2}
Let $A$ be a subnormal operator on a Hilbert space $X$ and assume $\sigma(A)\subset S_\theta$, for some
$\theta \in [0,\pi/2)$, and $B\in L(\mathbb C^m,X_{-1})$. Then:
\begin{itemize}
	\item[(i)]   $A$ generates a bounded analytic $C_0$-semigroup  of contractions $(T(t))_{t\ge 0}$,
	\item[(ii)]  $B$  is $\infty$-admissible for  $(T(t))_{t\ge 0}$.
\end{itemize}
Moreover, if in addition $0\not\in \sigma(A)$, then
\begin{itemize}
	\item[(iii)]  $A$ generates an exponentially stable semigroup and the system \eqref{InfiniteDim2} is ISS.
\end{itemize}
\end{proposition}
\begin{proof}
Assertion (i) follows from Lemma \ref{lem:Domains_of_subnormal_operators},
the fact that normal operators $N$ with $\sigma(N)\subset S_\theta$
generate  bounded analytic $C_0$-semigroups, see \cite[Corollary
II.4.7]{EnN00}, and from the observation that $A$ is the restriction of $N$ to an invariant subspace. 
The assertion (ii) has been proved in \cite{JSZ17}. 
Finally, $A$ generates an exponentially stable semigroup since for analytic semigroups the spectral bound equals the growth bound \cite[Corollary IV.3.12]{EnN00}, and ISS follows as $B$ is assumed to be $\infty$-admissible, see Example~\ref{exam2b}.
\end{proof}

We have the following important inequality for the subnormal operators:
\begin{proposition}\label{propsub}
Let $A$ be a subnormal operator on a Hilbert space $X$ satisfying $\sigma(A)\subset S_\theta$, for some
$\theta \in [0,\pi/2)$. Then  for  $\delta\ge 1-2\cos^2 \theta$ we have 
\begin{equation}\label{eq:key}
\re\, \langle x, A^2x\rangle_X + \delta \|Ax\|_X^2\ge 0, \qquad x\in D(A^2).
\end{equation}
\end{proposition}
\begin{proof}
Expanding \eqref{eq:key} we obtain the equivalent assertion
\begin{equation}\label{eq:key2}
 \re\, \langle \nu x, P \overline\nu x \rangle_{L^2(\mu)} + \delta \|\nu x\|_{L^2(\mu)}^2 \ge 0,
\end{equation}
and we note that $\langle \nu x, P \overline\nu x \rangle_{L^2(\mu)} = \langle \nu x,   \overline\nu x \rangle_{L^2(\mu)}= \langle \nu^2 x,x\rangle_{L^2(\mu)}$.
The left hand side of \eqref{eq:key2} is
\[
\langle (  \re\, \nu^2 + \delta |\nu|^2) x, x\rangle_{L^2(\mu)}
=
\langle (2(  \re\,  \nu)^2 + (\delta-1) |\nu|^2 ) x,x\rangle_{L^2(\mu)}.
\]
As the essential range of $\nu$ lies in $\sigma(A)$,  we have 
by sectoriality 
\[
2 ( \re\, \nu)^2 \ge 2 \cos^2\theta |\nu|^2
\]
and hence 
\[
\langle (2(  \re\,  \nu)^2 + (\delta-1) |\nu|^2 ) x,x\rangle_{L^2(\mu)} \ge 0,
\]
for $\delta\ge 1-2\cos^2\theta$.
\end{proof}

Now we can derive a converse ISS Lyapunov theorem for a broad class of systems with subnormal generators:

\begin{corollary}\label{corcon}
\label{cor:Converse_Theorem_Subnormal_Generators}
 Let $A$
be  a subnormal operator on a Hilbert space $X$ satisfying $\sigma(A)\subset S_\theta\backslash\{0\}$, for some
$\theta \in [0,\pi/2)$. Further, let $B\in L(\mathbb C^m,X_{-1})$ and let $\Uc:=L^\infty(\R_+,\C^m)$. 
Then
\begin{equation}\label{Lyap2}
 V(x) := - \re\, \langle A^{-1}x,x\rangle_X
 \end{equation}
is an ISS Lyapunov function for \eqref{InfiniteDim2} satisfying
\[  \dot{V}_u(x) \le - c_1 \|x_0\|_X^2 + c_2 \|u\|^2_\infty \]
for some constants $c_1, c_2>0 $ and all $x_0\in X$ and $u\in \Uc$.
\end{corollary}

\begin{proof}
By Proposition~\ref{prop2}, $A$ generates an exponentially stable and analytic $C_0$-semigroup  of contractions $(T(t))_{t\ge 0}$ and $B$  is $\infty$-admissible for  $(T(t))_{t\ge 0}$.
Further, Lemma~\ref{lem:Domains_of_subnormal_operators} guarantees that $D(A) \subset D(A^*)$.
As $A$ generates a contraction semigroup, the Lumer-Phillips theorem ensures that the operator $A$ is dissipative 
(that is, $\re\langle Ax,x\rangle_X \le 0$ for $x\in D(A)$). This together with $\sigma(A)\subset S_\theta\backslash\{0\}$  implies 
$\re\langle Ax,x\rangle_X < 0$ for $x\in D(A)\backslash\{0\}$.

%

Furthermore, as $0 \in\rho(A)$, for all $y \in D(A)$ there is $x \in D(A^2)$ so that $y = Ax$ and applying 
Proposition~\ref{propsub}, we obtain
\begin{eqnarray*}
0 \leq \re\, \langle x, A^2x\rangle_X + \delta \|Ax\|_X^2 = \re\, \langle A^*A^{-1}y, y\rangle_X + \delta \|y\|_X^2.
\end{eqnarray*}
This shows \eqref{eqn:a2}.

Hence all assumptions of Proposition~\ref{prop:Conv_ISS_LF_Theorem_LinOp} are satisfied, and application of 
Proposition~\ref{prop:Conv_ISS_LF_Theorem_LinOp} shows the claim.
\end{proof}

\begin{remark}
The above corollary also holds if we replace $B\in L(\mathbb C^m,X_{-1})$ by an $\infty$-admissible $B\in L(U,X_{-1})$, where $U$ is a Hilbert space.
\end{remark}

\subsection{ISS Lyapunov functions for input-to-state stable diagonal systems}
\label{sec:Example}

Consider a linear system \eqref{InfiniteDim2} with the state space
\[
X=l_2(\N):=\Big\{ x=\{x_k\}_{k=1}^{\infty}: \|x\|_{X}=\Big(\sum_{k=1}^{\infty} |x_k|^2 \Big)^{1/2}< \infty  \Big\}.
\] 
endowed in the usual way with the scalar product $\lel  \cdot,\cdot \rir_{l_2}$.
Let $U:=\R$.

Consider an operator $A:X\to X$, defined by $Ae_k=- \lambda_k e_k$, where $e_k$ is the $k$-th unit
vector of $l_2(\N)$ and $\lambda_k \in \R$ with $\lambda_k<\lambda_{k+1}$ for all $k$, $\lambda_1>\varepsilon>0$ and $\lambda_k \to \infty$  as $k \to \infty$. 
The operator $A$ can be represented using the spectral decomposition
\begin{eqnarray}\label{a1}
Ax:=\sum_{k=1}^\infty - \lambda_k \lel  x,e_k \rir_{l_2} e_k, \quad x \in D(A),
\label{eq:Generator_diagonal_semigroups}
\end{eqnarray}
with
\begin{eqnarray}\label{a2}
D(A) = \{x \in l_2(\N): \sum_{k=1}^\infty - \lambda_k \lel  x,e_k \rir_{l_2} e_k \text{ converges}\}.
\label{eq:Domain_generator_diagonal_semigroups}
\end{eqnarray}

We have the following result:
\begin{proposition}
\label{prop:Converse_Lyapunov_Theorem_Diagonal_System} 
Let $A$ be given by \eqref{a1}-\eqref{a2} and  $B \in L(\mathbb C^m,X_{-1})$. Then \eqref{InfiniteDim2} is  ISS and 
a non-coercive ISS Lyapunov function for \eqref{InfiniteDim2} can be chosen as
\begin{eqnarray}
V(x):=\sum_{k=1}^\infty \frac{1}{\lambda_k} \lel  x,e_k \rir_{l_2}^2.
\label{eq:ISS_LF_diagonal_semigroups}
\end{eqnarray}
\end{proposition}

\begin{proof}
By assumptions, the operator  $A$ is self-adjoint with  $\sigma(A) \subset (-\infty,0)$. Thus, the assumptions of Corollary~\ref{corcon} are satisfied. 
Moreover, the inverse of $A$ is given by
\begin{eqnarray}
A^{-1}x:=\sum_{k=1}^\infty - \frac{1}{\lambda_k} \lel  x,e_k \rir_{l_2} e_k,
\label{eq:Inverse_Generator_diagonal_semigroups}
\end{eqnarray}
and thus the Lyapunov function \eqref{Lyap2} has the form \eqref{eq:ISS_LF_diagonal_semigroups}.
\end{proof}


%
%
It is easy to see that $P$ (as well as the corresponding ISS Lyapunov function $V$) is not coercive since $\lambda_k \to \infty$ as $k \to \infty$. 

\subsection{ISS Lyapunov functions for a heat equation with Dirichlet boundary input}

It is well-known that a classical heat equation with Dirichlet boundary inputs is ISS, which has been verified by means of several different methods: \cite{JNP18, KaK16b, MKK19}. However, no constructions for ISS Lyapunov functions have been proposed. 
In the next example, we show that using Theorem~\ref{thm:Gen_ISS_LF_Construction},
 one can construct a non-coercive ISS Lyapunov function for this system.
\begin{example}
\label{ex1}
Let us consider the following boundary control system given by the one-dimensional heat equation on the spatial domain $[0,1]$ with Dirichlet boundary control at the point $1$,
\begin{align*}
x_t(\xi,t)={}&a x_{\xi\xi}(\xi,t), \quad \xi\in(0,1),~ t>0,\\
  x(0,t)={}&0,  \quad x(1,t)=u(t), \quad t>0,\\
x(\xi,0)={}&x_{0}(\xi),
\end{align*}
where $a>0$. We refer the reader to \cite[Chapter 10]{TuW09} for the definition and properties of boundary control systems.
We choose
 $X:=L^{2}(0,1)$, $U:=\mathbb C$ and  $\Uc:=L^\infty(\R_+,\C)$.
 Every boundary control system can be equivalently written in the form
 \begin{align*}
 \dot{x}(t)=Ax(t)+Bu(t),
 \end{align*}
 where $A$ generates a $C_0$-semigroup on $X$ and $B\in L(U, X_{-1})$, see \cite[Proposition 10.1.2 and Remark 10.1.4]{TuW09}.
 For  the one-dimensional heat equation on the spatial domain $[0,1]$ with Dirichlet boundary conditions the operator $A$ is given by
\begin{align*}
Af={}af'', \quad f\in D(A):={}\left\{f\in H^{2}(0,1):   
f(0)= f(1)=0\right\}.
\end{align*}
Here $H^2(0,1)$ denotes the Sobolev space of functions $f \in L^2(0,1)$, which have weak derivatives of order $\leq 2$, all of which belong to $L^2(0,1)$.
It is well-known that $A$ is a self-adjoint operator on $X$ generating an exponentially stable analytic $C_0$-semigroup on $X$. 
By \cite[Theorem 1 and Proposition 5]{JSZ17}, we get that $B\in L(U, X_{-1})$ is $\infty$-admissible, for every $x_0\in X$ and $u\in L^\infty(0,\infty)$ the corresponding mild solution is continuous with respect to time, and $\kappa(0)=0$.
In \cite{JNP18}, the following ISS-estimates have been shown for every $x_0\in X$, $u\in \Uc$, $p> 2$ and some $c=c(p)>0$. 
\begin{align*}
\|x(t)\|_{L^2(0,1)}&\le {\rm{e}}^{-a\pi^{2} t}\|x_0\|_{L^2(0,1)} +\frac{1}{\sqrt{3}} \|u\|_{L^\infty(0,t)},\\
\|x(t)\|_{L^2(0,1)}&\le {\rm{e}}^{-a\pi^{2} t}\|x_0\|_{L^2(0,1)} +c \left( \int_0^t |u(s)|^pds \right)^{1/p},
\end{align*}
Direct application of Corollary~\ref{cor:Converse_Theorem_Subnormal_Generators}
shows that 
\begin{eqnarray*}
 V(x) &=& -  \langle A^{-1}x,x\rangle_X = \int_0^1 \left(\int_\xi^1  (\xi -\tau) x(\tau) d\tau \right)\overline{x(\xi)} d\xi 
 \end{eqnarray*}
is a non-coercive ISS Lyapunov function for the  one-dimensional heat equation on the spatial domain $[0,1]$ with Dirichlet boundary control at the point $1$.
In turn, the constructed non-coercive ISS Lyapunov function implies ISS of the considered system.
\end{example}

\section{Conclusions}

In this paper, we have investigated the question to what extent the
existence of a non-coercive ISS Lyapunov function implies that a forward complete system is input-to-state stable (ISS). 
It was shown that norm-to-integral ISS follows from the
existence of such Lyapunov functions for a large
class of systems. Furthermore, we show that norm-to-integral ISS is equivalent to ISS for the systems possessing the continuity of the flow map near the equilibrium and boundedness of finite-time reachability sets.
 These assumptions are related to questions of the richness of the possible
dynamics both close to the origin and in the large. 

Non-coercive Lyapunov functions are to some extent natural in infinite dimensions. Already Datko's construction of quadratic
Lyapunov functions $V(x)=\scalp{Px}{x}$ for exponentially stable linear systems on Hilbert space generally leads to non-coercive Lyapunov functions. 
In this work, we show that under some additional conditions, which relate the infinitesimal generator of a semigroup and an operator $P$, this function $V$ is a non-coercive ISS Lyapunov function for a linear system with any $\infty$-admissible input operator.  
Furthermore, we have shown in this paper that for broad classes of linear systems with unbounded input operators (including analytic systems with subnormal generators) the construction of Lyapunov functions using the resolvent at $0$ as an operator $P$ is a
natural choice and one that leads to noncoercive Lyapunov functions. 
As an example, we have constructed an ISS Lyapunov function for a heat equation with a Dirichlet boundary input, which seems to be the first construction of an ISS Lyapunov function for this system, which was widely studied by non-Lyapunov methods.

In a future work, we plan to extend the class of systems for which explicit
constructions are possible and to deepen our understanding of noncoercive
ISS Lyapunov functions.

\vspace{-5mm}

\bibliographystyle{abbrv}
\bibliography{Mir_LitList}

\end{document}